\documentclass[12pt]{article}
\usepackage{graphicx}
\usepackage{xcolor}
\usepackage{amssymb,amsfonts,amsthm,amsmath,calligra,tikz}
\usepackage[utf8]{inputenc}
\usepackage{accents}
\usepackage{slashed}
\usepackage{yfonts}
\usepackage{mathrsfs,pifont}
\theoremstyle{definition}
\usepackage[all]{xy}
\usepackage{enumitem}

\usepackage[style=alphabetic]{biblatex}
\def\be{\begin{eqnarray}}
\def\ee{\end{eqnarray}}

\tikzstyle{vertex}=[circle, draw, inner sep=0pt, minimum size=6pt]

\def\matZ{{\mathbb{Z}}}
\def\matR{{\mathbb{R}}}
\def\matQ{{\mathbb{Q}}}
\def\matC{{\mathbb{C}}}
\def\matN{{\mathbb{N}}}

\newcommand{\bT}{\mathsf{T}}
\newcommand{\bK}{\mathsf{K}}

\newcommand{\Lie}{\mathrm{Lie}}

\let\bs\boldsymbol

\newcommand{\sslash}{\mathbin{/\mkern-6mu/}}

\def\zz{{\bs z}}

\def\aa{{\bs a}}

\theoremstyle{definition}
\newtheorem{Definition}{Definition}
\newtheorem{Proposition}{Proposition}
\newtheorem{Lemma}{Lemma}
\newtheorem{Corollary}{Corollary}
\newtheorem{Conjecture}{Conjecture}
\newtheorem{Theorem}{Theorem}

\newtheorem{Note}{Note}

\newcommand{\fC}{\mathfrak{C}} 

\makeatletter
\newcommand{\somespecialrotate}[3][]{%
\begingroup
\sbox\@tempboxa{#3}%
\@tempdima=.5\wd\@tempboxa
\sbox\@tempboxa{\rotatebox[#1]{#2}{\usebox\@tempboxa}}%
\advance\@tempdima by -.5\wd\@tempboxa
\mbox{\hskip\@tempdima\usebox\@tempboxa}%
\endgroup}
\linethickness{1 pt}

\def\qm {{{\textnormal{\textsf{QM}}}}}

\def \vss {\widehat{{\O}}_{{\rm{vir}}}}

\def\zz{{\bs z}}
\def\dd{{\bs d}}

\def\O {{\mathcal{O}}}

\begin{document}
\title{Capped vertex with descendants for zero dimensional $A_{\infty}$ quiver varieties}
\author{H. Dinkins and A. Smirnov}
\date{}
\maketitle
\thispagestyle{empty}
	
\begin{abstract}
In this paper, we study the capped vertex functions associated to certain zero-dimensional type-$A$ Nakajima quiver varieties. The insertion of descendants into the vertex functions can be expressed by the Macdonald operators, which leads to explicit combinatorial formulas for the capped vertex functions. 

We determine the monodromy of the vertex functions and show that it coincides with the elliptic R-matrix of symplectic dual variety. We apply our results to give the vertex functions and the characters of the tautological bundles on the quiver varieties formed from  arbitrary stability conditions.
\end{abstract}
	
\setcounter{tocdepth}{2}

\section{Introduction}
Our main objects of study in this paper are certain $K$-theoretic enumerative invariants of Nakajima quiver varieties known as vertex functions, see Section 7 in \cite{pcmilect} and Section 2 below. Let $X$ be a Nakajima quiver variety, see \cite{GinzburgLectures}, \cite{Nak1}, or Section 2 in \cite{MO}  for an introduction. Vertex functions for $X$ come in two flavors, capped and uncapped. The vertex functions are defined by an equivariant count of quasimaps to $X$, and the type of vertex function is determined by whether the moduli space of quasimaps is considered with a nonsingular condition or a relative condition, see Section 2 below. 

One can consider either type of vertex function with descendants inserted, and these give a collection of natural classes in the $K$-theory of $X$. Starting with a vector bundle on $X$ and an evaluation map on the appropriate moduli space of quasimaps, we can pullback the vector bundle under the evaluation map to obtain a class on the quasimap moduli space. The insertion of a descendant into a vertex function refers to the quasimap count obtained by tensoring the structure sheaf of the quasimap moduli space with a class obtained in this way. The $K$-theory of Nakajima quiver varieties is generated by a collection of tautological vector bundles, one for each vertex of the quiver, and any of these gives rise to a descendant that can be inserted into a vertex function.

For quiver varieties arising from type-$A$ quivers, there are known procedures for computing the vertex functions as a power series in the K\"ahler parameters, see Section 1 in \cite{OkBethe}, Section 4.5 in \cite{Pushk1}, and Section 2.4 \cite{dinksmir2}. In this paper, we restrict our attention to the capped and uncapped vertex functions with descendants for zero-dimensional type-$A$ quiver varieties. Such varieties are indexed by partitions, and we denote them by $X_{\lambda}$ for a partition $\lambda$. 

Our main result is that the insertion of descendants into the uncapped vertex function can be realized by the action of the difference operators of the trigonometric Ruijsenaars-Schneider model (also known as the Macdonald difference operators). Let $\textbf{V}_{\lambda}(\zz)$ be the uncapped vertex function for $X_{\lambda}$, and let $\textbf{V}^{(\tau)}_{\lambda}(\zz)$ be the uncapped vertex function with descendant $\tau$, then the Theorem  \ref{tRS} reads:
\begin{equation}\label{eq}
T(\tau) \textbf{V}_{\lambda}(\zz)= \textbf{V}^{(\tau)}_{\lambda}(\zz)
\end{equation}
where $T(\tau)$ is the Macdonald difference operator associated with $\tau$.

In \cite{dinksmir2}, we proved
\be \label{vacvect}
\textbf{V}_{\lambda}(\zz)= \prod\limits_{\square \in \lambda} \prod\limits_{i=0}^{\infty} \dfrac{1-\hbar  z_{\square} q^{i}}{1- z_{\square} q^{i}}
\ee
where $z_{\square}$ denotes a certain monomial in the K\"ahler parameters depending on the box $\square$ in the Young diagram for $\lambda$, see (\ref{prform}) below. 

This result allows us to explicitly compute the vertex functions with descendant insertions. For instance, if $\mathcal{V}_n$ is the tautological vector bundle on $X_{\lambda}$ corresponding to $n$th vertex of the quiver and $\tau_{n,r}=\bigwedge^r \mathcal{V}_n$, then from  (\ref{eq}) and (\ref{vacvect}), we obtain the following rational function for the capped vertex function with descendant $\tau_{n,r}$:
$$
 \hat{\textbf{V}}_{\lambda}^{(\tau_{n,r})}(\zz)= \sum_{\substack{I \subset C_{\lambda}(n) \\ |I|=r }} \prod_{\substack{\square \in I \\ \square' \in C_{\lambda}(n)\setminus I}} \frac{\hbar \zeta_{\square'} -  \zeta_{\square}}{\zeta_{\square'} - \zeta_{\square}} \prod_{\square\in I}\prod_{\square' \in S_{\lambda}(\square)} \frac{1-z_{\square'}}{1-\hbar z_{\square'}}
$$
where $C_{\lambda}(n)$ and $S_{\lambda}(\square)$ denote certain subsets of boxes in $\lambda$, and $\zeta_{\square}$ are monomials in $\zeta_i$, related to the K\"ahler parameters by $z_i=\frac{\zeta_{i-1}}{\zeta_{i}}$.

In Sections 5 and 6, we give an application of our formula to quiver varieties with non-canonical stability conditions. In general, Nakajima quiver varieties are defined as geometric invariant theory quotients and thus depend on a choice of stability condition $\theta \in \Lie_{\mathbb{R}}(\bK)$, where $\bK$ is the K\"ahler torus $\bK:= (\mathbb{C}^\times)^{|I|}$ and $I$ is the vertex set of the quiver. The varieties $X_{\lambda}$ are determined by the positive stability condition, see (\ref{variety}) below. In the literature, explicit computations involving Nakajima quiver varieties almost always consider only the positive and negative stability conditions.

As the stability condition varies, the varieties obtained from them change only when crossing certain hyperplanes. This gives rise to a collection of cones in $\Lie_{\matR}(\bK)$. The toric compactification $\overline{\bK}$ of $\bK$ given by the fan generated by these cones is known as the K\"ahler moduli space. 

The vertex function of $X_{\lambda}$ is the solution of a $q$-difference equation (see \cite{OS}), and it is expected that the vertex functions of the varieties given by the same dimension data as $X_{\lambda}$, but with different stability conditions, solve the same $q$-difference equation. Furthermore, the vertex function corresponding to a choice of stability condition gives a solution of this $q$-difference equation holomorphic in a neighborhood of the limit point on the K\"ahler moduli space corresponding to this stability condition. 

By studying the explicit form of the $q$-difference equation in Section 5, it is straightforward to give a formula for the solution holomorphic in a neighborhood of an arbitrary limit point of the K\"ahler moduli space. For the reasons explained above, such solutions are expected to coincide with the vertex functions of the quiver variety with the appropriate stability condition. As further evidence of this expected correspondence, we examine the monodromy of the $q$-difference equation and verify that this agrees, up to a constant, with the elliptic R-matrix of the symplectic dual variety, see~\cite{AOElliptic}. 

Putting all this together, we start with the capped vertex function with descendant $\tau$ for the variety with positive stability condition, and examine the limit as the K\"ahler parameters approach a limit point corresponding to a general stability condition $\theta$. This provides us with the character of $\tau$ on the quiver variety with identical dimension data as $X_{\lambda}$ and arbitrary stability parameter~$\theta$.

\section*{Acknowledgments}
A. Smirnov was supported by the Russian Science Foundation under grant 19-11-00062. The authors would also like to thank Peter Koroteev for pointing out the connection between the vertex functions of the cotangent bundle of the flag variety and the varieties studied here.

\section{Quasimaps and Vertex Functions}
\subsection{}
Let $\lambda$ be a Young diagram rotated  by $45^{\circ}$ as in Figure \ref{yng0}.  Let $\textsf{v}_i \in  \matN$, $i\in \matZ$ denote the number of boxes in the $i$th vertical column, as oriented in the Figure. We assume that $i=0$ corresponds to the column which contains the corner  box of $\lambda$. 
\begin{figure}[ht]
\centering
\begin{tikzpicture}[roundnode/.style={circle, draw=black, very thick, minimum size=5mm},squarednode/.style={rectangle, draw=black, very thick, minimum size=5mm}] 
\draw[ultra thick] (-5,5)--(0,0) -- (4,4);
\draw[ultra thick] (-4,6)--(1,1);
\draw[ultra thick] (-5,5)--(-4,6);
\draw[ultra thick] (-2,6)--(2,2);
\draw[ultra thick] (-4,4)--(-2,6);
\draw[ultra thick]  (0,6)--(3,3);
\draw[ultra thick] (2,6)--(4,4);
\draw[ultra thick] (-1,1)--(3,5);
\draw[ultra thick] (-2,2)--(2,6);
\draw[ultra thick] (-3,3)--(0,6);
\draw[ultra thick] (4,4)--(2,6);
\end{tikzpicture}
\caption{The partition $\lambda=(5,4,3,2)$ rotated by $45^{\circ}$ and $\textsf{v}=(\textsf{v}_i)=(\ldots,0,0,1,1,2,2,3,2,2,1,0,0,\ldots)$.} \label{yng0}
\end{figure}
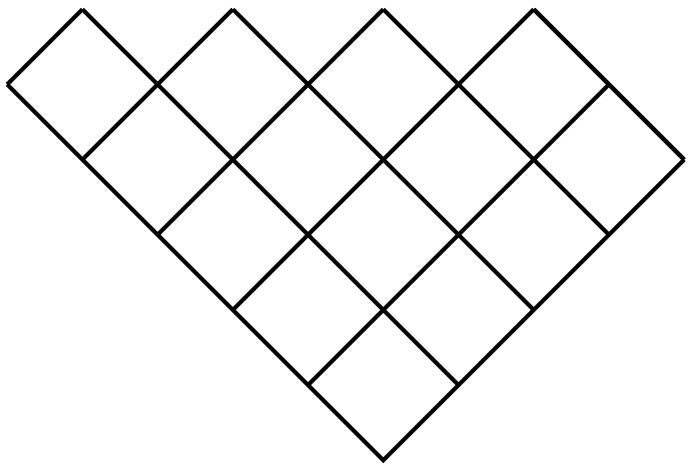
Let $\textsf{v}=(\textsf{v}_i)$ and $\textsf{w}_{i}=\delta_{i,0}$. Let  
\be \label{variety}
X_{\lambda}:={\cal{M}}(\textsf{v}, \textsf{w})
\ee
denote the $A_{\infty}$ Nakajima quiver variety defined by these data, with stability condition given by the character 
$$
\theta_+: (g_i) \to \prod_i\det (g_i).
$$
We will refer to this stability condition as the {\it canonical} 
or {\it positive} stability condition. One of the goals of this paper is to analyze the enumerative invariants (vertex functions) of quiver varieties arising for generic $\theta=(\theta_1,\dots,\theta_{|I|})$ stability conditions corresponding to characters
$$
\theta: (g_i) \to \prod_i\det (g_i)^{\theta_i}.
$$

Throughout, we will assume that the $A_{\infty}$ quiver is oriented with arrows pointing to the right. Nakajima quiver varieties come equipped with a natural action of a torus $\bT$ and a collection of $\bT$-equivariant vector bundles which we call $\mathcal{V}_{i}$, $i \in \mathbb{Z}$. We denote by $\hbar$ the weight of the $\bT$-module $\mathbb{C} \omega$, where $\omega$ is the symplectic form on $X_{\lambda}$.

\subsection{}
To define the vertex functions, we need to study moduli spaces of stable quasimaps from $\mathbb{P}^1$ to a Nakajima quiver variety $X$, as introduced in \cite{qm}. We review the main objects of study in the case of Nakajima quiver varieties, see Section 6 of \cite{pcmilect} and Section 2 of \cite{Pushk1}.

The definition of a quasimap to a variety $X$ requires a presentation of $X$ as a geometric invariant theory quotient. For a Nakajima quiver variety arising from a quiver $Q$ with vertex set $I$ and dimensions $\mathsf{v},\mathsf{w}$, this takes the form
$$
X:= \mu^{-1}(0)\sslash_{\theta} G_{\mathsf{v}} = \mu^{-1}(0)^{\theta-s} / G_{\mathsf{v}}
$$
where $\mu$ is the moment map associated to the $G_{\mathsf{v}}:=\prod_{i \in I} GL(\mathsf{v}_{i})$ action on $T^*Rep_Q(\mathsf{v},\mathsf{w})$, the cotangent bundle of the space of framed representations of $Q$ and $\mu^{-1}(0)^{\theta-s}$ is the intersection of $\mu^{-1}(0)$ with the stable points defined by a choice of  $\theta$ (\cite{GinzburgLectures}). 

In the context of the varieties $X_{\lambda}$, this data looks as follows. The dimension vectors correspond to vector spaces $V_{i}$ and the space $T^*Rep_Q(\mathsf{v},\mathsf{w})$ consists of 4-tuples $(A,B,I,J)$ so that
$$
A=\bigoplus_{i\in \matZ} A_i, \ \ B=\bigoplus_{i\in \matZ} B_i, \ \ \ A_i\in \mathrm{Hom}(V_{i},V_{i+1}), \ \ B_i\in \mathrm{Hom}(V_{i+1},V_{i}).
$$
and 
$$
I\in V_0 \ \ J\in V_0^*
$$
The moment map is $\mu(A,B,I,J)=[A,B]+ IJ$ and the tuple $(A,B,I,J)$ is $\theta_{+}$-stable if and only if 
\begin{equation}\label{stab}
    \mathbb{C}\langle A,B \rangle I= \bigoplus_{i\in \mathbb{Z}} V_i
\end{equation}
where $\mathbb{C}\langle A,B \rangle$ denotes the ring of (noncommutative) polynomials in $A,B$.

For generic choice of $\theta$ the condition (\ref{stab}) should be substituted by more complicated conditions described in Proposition 5.1.5 in \cite{GinzburgLectures}.

\subsection{}
We  recall some facts about quasimaps to a quiver variety $X$. For more details, see \cite{qm} and \cite{pcmilect} Sections 4-6.

\begin{Definition}\label{qm}
A stable genus zero quasimap to $X$ relative to $p_1,\ldots,p_m$ is given by the following data
$$
(C,p_1',\ldots,p_m',P,f,\pi)
$$
where 
\begin{itemize}
    \item $C$ is a genus zero connected curve with at worst nodal singularities and the $p_1', \ldots, p_m'$ are nonsingular points of $C$.
    \item $P$ is a principal $G_{\mathsf{v}}$ bundle over $C$.
    \item $f$ is a section of the bundle $P\times_{G_{\mathsf{v}}} T^*Rep_Q(\mathsf{v},\mathsf{w})$ satisfying $\mu=0$.
    \item $\pi: C \to D$ is a regular map.
\end{itemize}
satisfying the following conditions
\begin{enumerate}
    \item There is a distinguished component $C_0$ of $C$ so that $\pi$ restricts to an isomorphsism $\pi: C_0 \cong D$ and $\pi(C\setminus C_0)$ is zero dimensional (possibly empty).
    \item $\pi(p_i')=p_i$.
    \item $f(p)$ is stable in the sense of (\ref{stab}) for all but a finite set of points disjoint from $p_1',\ldots,p_m'$ and the nodes of $C$.
    \item The line bundle $\omega_{\widetilde{C}}\left(\sum_{i} p_i' + \sum_j q_j \right) \otimes \mathcal{L}_{\theta}^{\epsilon}$ is ample for every rational $\epsilon>0$, where $\mathcal{L}_{\theta}= P \times_{G_{\mathsf{v}}} \mathbb{C}_{\theta}$, $\widetilde{C}$ is the closure of $C\setminus C_0$, $q_j$ are the nodes of $\widetilde{C}$, and $\mathbb{C}_{\theta}$ is the one dimensional $G_{\mathsf{v}}$-module defined by the stability condition $\theta$.
\end{enumerate}
\end{Definition}

\begin{figure}[ht]
\centering
\includegraphics[width=8cm]{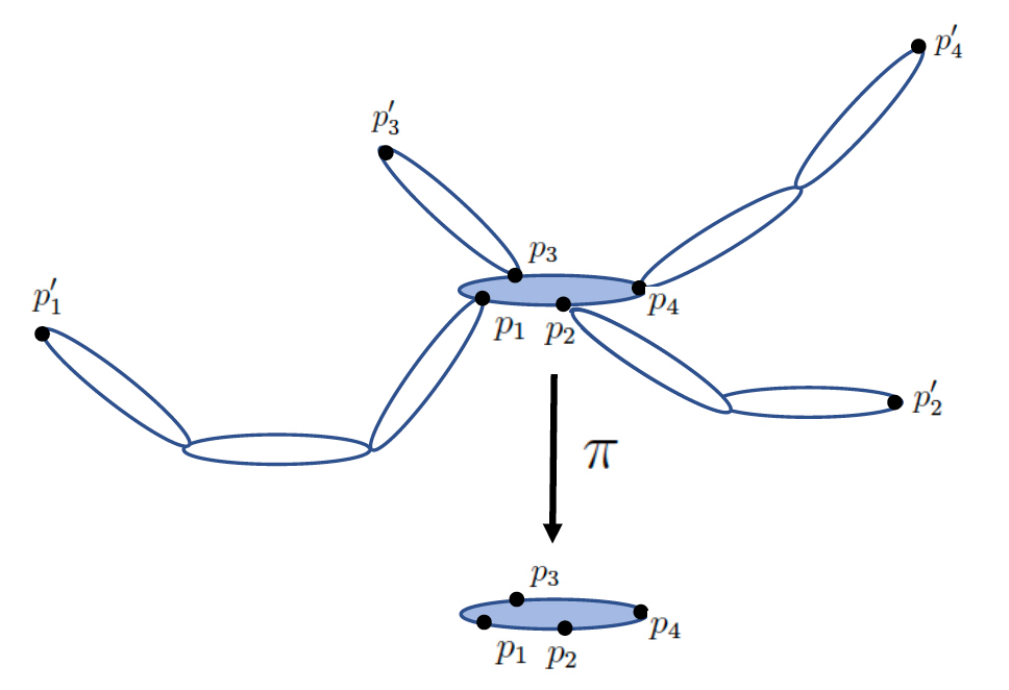}
\caption{An example of the domain of a relative quasimap with four marked points. A chain of rational curves is attached to each point $p_i$, and condition 4 implies that the last component of each chain has a marked point $p_i'$. The map $\pi$ collapses each chain to a single point.}
\end{figure}

\begin{Definition}
A relative quasimap $(C,p_1',\ldots,p_m',P,f,\pi)$ is nonsingular at $p\in C$ if $f(p)$ is stable in the sense of (\ref{stab}). In this case, $f(p)$ gives a point in the quiver variety.
\end{Definition}

\begin{Definition}
The degree of a quasimap $(C,p_1',\ldots,p_m',P,f,\pi)$ is the tuple $\dd=(d_i)_{i \in \mathbb{Z}}$ where $d_i$ is the degree of the rank $\mathsf{v}_i$ vector bundle $P\times_{G_{\mathsf{v}}} V_i \to C$.
\end{Definition}

\begin{Theorem}(\cite{qm} Theorem 7.2.2)
The stack $\qm^{\dd}_{\textrm{relative}\, p_1, \ldots, p_m}$ parameterizing the data of stable genus zero quasimaps to $X$ is a Deligne-Mumford stack of finite type with a perfect obstruction theory.
\end{Theorem}

\begin{Definition}Let $\qm^{\dd}_{\textrm{nonsing}\, p_1,\ldots,p_m}$ be the stack parameterizing the data of degree $\dd$ quasimaps to $X$ relative to $p_1,
\ldots,p_m$ such that $C\cong D \cong \mathbb{P}^1$. For such a quasimap, most of the conditions in Definition \ref{qm} become trivially satisfied.
\end{Definition}
Restricting the obstruction theory of $\qm^{\dd}_{\textrm{relative}\, p_1, \ldots, p_m}$ gives a perfect obstruction theory on $\qm^{\dd}_{\textrm{nonsing}\, p_1,\ldots,p_m}$. The symmetrized virtual structure sheaf on such a space will be denoted by $\vss^{\dd}$, with the context determining exactly which quasimap space we are considering.

Given a quasimap $(C,p_1',\ldots,p_m',P,f,\pi)$ and $p\in C$, there is an evaluation map to the quotient stack: 
$$
\mathrm{ev}_p(C,p_1',\ldots,p_m',P,f,\pi)=f(p)\in [\mu^{-1}(0)/G_{\mathsf{v}}]
$$

Given a Schur functor $\tau$ in the tautological bundles on $X_{\lambda}$, let $\tau_{\text{stack}}$ be the associated $K$-theory class on $[\mu^{-1}(0)/G_{\mathsf{v}}]$. Then we can define an induced $K$-theory class on $\qm^{\dd}_{\textrm{relative}\, p_1, \ldots, p_m}$:
\be \label{stackt}
\tau|_p:= \mathrm{ev}_p^*(\tau_{\text{stack}})
\ee
\subsection{}
The action of the torus $\bT$ on a quiver variety $X$ and of $\mathbb{C}^\times_q$ on $\mathbb{P}^1$ induce an action of $\bT\times \mathbb{C}^\times_q$ on quasimaps to $X$. Let $p_1=0$ and $p_2=\infty$ in $\mathbb{P}^1$. In what follows, we will denote $\zz^{\dd}=\prod_{i} z_i^{d_i}$ and use this notation to keep track of the degree of quasimaps. The variables $z_i$ are known as the K\"ahler parameters, and are characters of the K\"ahler torus
$$
\bK:= \left(\mathbb{C}^\times\right)^{|I|}
$$
where $I$ is the vertex set of the quiver.

The evaluation maps on relative quasimaps are proper (\cite{pcmilect} Section 7.4), and thus we can make the following definition.
\begin{Definition}
The capped vertex function with descendant $\tau$ inserted at $p_1$ is the formal power series
$$
\hat{\textbf{V}}^{(\tau)}(\zz)=\sum\limits_{\dd} \textrm{ev}_{p_2, *}(\vss^{\dd}\otimes \tau|_{p_1},\qm^{\dd}_{\textrm{relative}\, p_2} ) \zz^\dd  \in K_{\bT}(X)[[\zz]]
$$
where $\vss^{\dd}$ is the symmetrized virtual structure sheaf on $\qm^{\dd}_{\textrm{relative}\, p_2}$.
\end{Definition}
While the evaluation map $\mathrm{ev}_{p_2}$ on $\qm^{\dd}_{\text{nonsing}\, p_2}$ is not proper, the restriction to the $\mathbb{C}^{\times}_q$-fixed locus
$$
\mathrm{ev}_{p_2}: \left(\qm^{\dd}_{\text{nonsing}\, p_2}\right)^{\mathbb{C}^\times_q} \to X
$$
is (\cite{pcmilect} Section 7.2). Using equivariant localization, we can thus make the following definition.
\begin{Definition}
The bare vertex function with descendant $\tau$ inserted at $p_1$ is the formal power series
$$
\textbf{V}^{(\tau)}(\zz)=\sum\limits_{\dd} \textrm{ev}_{p_2, *}(\vss^{\dd}\otimes \tau|_{p_1},\qm^{\dd}_{\text{nonsing}\, p_2}) \zz^\dd  \in K_{\bT \times \mathbb{C}^\times_q}(X)_{loc}[[\zz]]
$$
where $\vss^{\dd}$ is the symmetrized virtual structure sheaf on $\qm^{\dd}_{\text{nonsing}\, p_2}$.
\end{Definition}

In what follows, we will omit the superscript $(\tau)$ in the bare vertex function when $\tau=1$.

\subsection{}

\begin{Definition} The capping operator is the formal series
$$\Psi(\zz)= \sum_{\dd} \textrm{ev}_{p_1,*} \otimes \textrm{ev}_{p_2,*}(\vss^\dd,\qm^{\dd}_{\substack{\textrm{relative}\,p_1 \\ \textrm{nonsing}\,p_2}}) \zz^\dd \in K_{\bT}^{\otimes 2}(X)_{loc}[[\zz]]$$
where $\vss^\dd$ denotes the symmetrized virtual structure sheaf on $\qm^{\dd}_{\substack{\textrm{relative}\,p_1 \\ \textrm{nonsing}\,p_2}}$
\end{Definition}
The standard pairing on equivariant $K$-theory
$$
(\mathcal{F}, \mathcal{G})= \chi(\mathcal{F}\otimes \mathcal{G})
$$
allows us to interpret $\Psi(\zz)$ as a linear map $$
\Phi(z): K_{\bT}(X)_{loc}[[\zz]] \to K_{\bT}(X)_{loc}[[\zz]]
$$
We have the following theorem:

\begin{Theorem}(\cite{pcmilect} Section 7.4)\label{capping}
The capping operator satisfies the equation
$$\hat{\textbf{V}}^{(\tau)}(\zz)= \Psi(\zz) \textbf{V}^{(\tau)}(\zz)$$
\end{Theorem}

\subsection{}
In the simplest situation of the zero-dimesnional quiver varieties $X_{\lambda}$, which we consider in the present paper, we have $K_{\bT}(X_{\lambda})_{loc}=\matQ(\hbar,q)$. Thus, in this case all the functions defined above are power series in the K\"ahler parameters with some rational coefficients in  $\hbar,q$:
$$
\hat{\textbf{V}}^{(\tau)}(\zz), \textbf{V}^{(\tau)}(\zz), \Psi(\zz)  \in \matQ(\hbar,q)[[\zz]]
$$
Let us denote $g(\zz)=\hat{\textbf{V}}^{(1)}(\zz)$. With this notation, it follows from the previous theorem that for $X_{\lambda}$ we have
$$
\Psi(\zz) = \frac{g(\zz)}{\textbf{V}^{(1)}(\zz)}
$$
and thus the capped vertex with descendent $\tau$ has the form:
$$
\hat{\textbf{V}}^{(\tau)}(\zz)=g(\zz) \dfrac{\textbf{V}^{(\tau)}(\zz)}{\textbf{V}^{(1)}(\zz)}
$$
We see that $g(\zz)$ appears as a normalization prefactor in the formulas for the capped vertex. Thus, it will be convenient to redefine the capped vertex function by normalizing as $\hat{\textbf{V}}^{(\tau)}(\zz)\to \hat{\textbf{V}}^{(\tau)}(\zz)/g(\zz)$.  

Let us note here that the function $g(\zz)$ can be computed explicitly.  It coincides with the multiplicative identity of the quasimap quantum K-theory ring, see Section 3.2 \cite{Pushk1}. In the case of $X_{\lambda}$, it is given by the gluing matrix ${\bf G}$. It can be shown \cite{Kononov2} that the gluing matrix of $X$ equals the zero slope $K$-theoretic $R$-matrix of the symplectic dual variety $X^{!}$. 
The K-theoretic version of Proposition \ref{Rmatprop} then gives:
$$
g(\zz)={\bf G}=\prod\limits_{\Box\in \lambda} \dfrac{1-z_{\Box}}{1-\hbar z_{\Box}}.
$$

\section{Vertex Functions for $X_{\lambda}$}
\subsection{}
Fix a partition $\lambda$, We do not distinguish between a partition and its Young diagram. Let $\mathsf{v}=(\mathsf{v}_i)$ be as in Section 2.1. We define the following functions for $\square \in \lambda$. Let $c_{\lambda}(\square)$ denote the content of $\square$, or the horizontal coordinate of $\square$ when the Young diagram of $\lambda$ is rotated as in Figure \ref{yng2}. We assume the corner box has content 0. Let $C_{\lambda}(n)$ be the set of all boxes of content $n$.

Let $h_{\lambda}(\square)$ denote the height of $\square$ in $\lambda$, normalized so that the heights of the boxes with content $i$ take all the values between 1 and $\mathsf{v}_i$. Let $a_{\lambda}(\square)$ and $l_{\lambda}(\square)$ denote the length of the arm and leg (not including $\square$) based at $\square$, respectively. If $\square$ has rectangular coordinates $(i,j)$, then 
$$l_{\lambda}(\square)=\lambda_i-j \quad \text{and} \quad
a_{\lambda}(\square)=\lambda'_j-i
$$
where $\lambda'$ is the transpose of $\lambda$.

Define
$$
\zeta_{\square}= \begin{cases} \left(\frac{\hbar}{q}\right)^{h_{\lambda}(\square)}\zeta_{c(\square)+a_{\lambda}(\square)} & \text{if} \ \ c(\square) \geq 0\\
\left(\frac{q}{\hbar}\right)^{h_{\lambda}(\square)} \zeta_{c(\square)-l_{\lambda}(\square)-1} & \text{if} \ \  c(\square)< 0
\end{cases}
$$
where $\zeta_i$ are a collection of variables related to the K\"ahler parameters by $z_i=\frac{\zeta_{i-1}}{\zeta_{i}}$.
Define the difference operator
$$
p_i f(\zeta_i):=f(q\zeta_i)
$$
where $f$ is some function of $\zeta_i$.
Then define
\begin{equation} \label{pbox}
p_{\square}=\begin{cases} \prod_{i=c(\square)+\mathsf{v}_{c(\square)}-h_{\lambda}(\square)}^{c(\square)+a_{\lambda}(\square)} p_i^{-1} & \text{if} \ \ c(\square)\geq 0 \\
\prod_{i=c(\square)-\mathsf{v}_{c(\square)}+h_{\lambda}(\square)-1}^{c(\square)+l_{\lambda}(\square)-1} p_i & \text{if} \ \  c(\square)<0
\end{cases}
\end{equation}

See Figure \ref{yng2} for an example. 

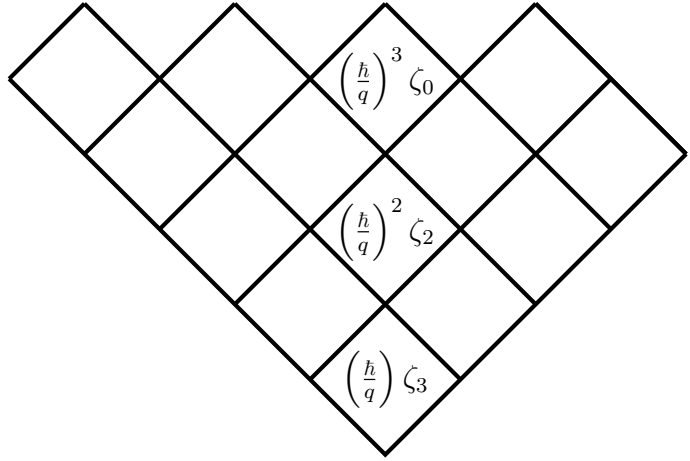
\begin{figure}[ht]
\centering
\begin{tikzpicture}[roundnode/.style={circle, draw=black, very thick, minimum size=5mm},squarednode/.style={rectangle, draw=black, very thick, minimum size=5mm}] 
\draw[ultra thick] (-5,5)--(0,0) -- (4,4);
\draw[ultra thick] (-4,6)--(1,1);
\draw[ultra thick] (-5,5)--(-4,6);
\draw[ultra thick] (-2,6)--(2,2);
\draw[ultra thick] (-4,4)--(-2,6);
\draw[ultra thick]  (0,6)--(3,3);
\draw[ultra thick] (2,6)--(4,4);
\draw[ultra thick] (-1,1)--(3,5);
\draw[ultra thick] (-2,2)--(2,6);
\draw[ultra thick] (-3,3)--(0,6);
\draw[ultra thick] (4,4)--(2,6);
\node[] at (0,5) {$\left(\frac{\hbar}{q}\right)^3 \zeta_{0}$};
\node[] at (0,3) {$\left(\frac{\hbar}{q}\right)^2 \zeta_{2}$};
\node[] at (0,1) {$\left(\frac{\hbar}{q}\right) \zeta_{3}$};
\end{tikzpicture}
\caption{The values of $\zeta_{\square}$ are shown inside the boxes of $C_{\lambda}(0)$. If $\square$ is the box with content 0 and height 1, then $a_{\lambda}(\square)=3$ and $l_{\lambda}(\square)=4$. Also, $p_{\square}=\prod_{i=0+3-1}^{0+3} p_i^{-1}=p_{2}^{-1}p_{3}^{-1}$.} \label{yng2}
\end{figure}

\subsection{}
Fix $n,r\in \mathbb{Z}$ so that $1\leq r \leq \mathsf{v}_n$.


We define the following difference operator:
$$T^{n,r}_{\lambda}= \hbar^{r(r-1)/2+\beta(n)}\sum_{\substack{I \subset C_{\lambda}(n) \\ |I|=r }} \prod_{\substack{\square \in I \\ \square' \in C_{\lambda}(n) \setminus I}} \frac{\hbar \zeta_{\square'} -  \zeta_{\square}}{\zeta_{\square'} - \zeta_{\square}} \prod_{i \in I} p_{\square}$$
where $\beta(n)=|n|$ if $n<0$ and $0$ otherwise.
We note that up to normalization and relabeling of the variables, the operators $T_{\lambda}^{n,r}$ are exactly the Macdonald difference operators, also known as the difference operators of the  trigonometric Ruijsenaars-Schneider model, see \cite{mac} and \cite{tRSKor}. In particular the $q$-difference operators $T^{n,r}_{\lambda}$ commute with each other. We denote by 
$$
\textsf{RS}=\matQ[T^{n,r}_{\lambda}]_{n,r\in \matZ}
$$
the commutative ring they generate.

\subsection{}
Let us recall that the descendent insertions of the vertex functions are labeled by elements (\ref{stackt}) of the ring:
$$
 K( [\mu^{-1}(0)/G_{\textsf{v}}] ) \cong \textrm{char}(G_{\textsf{v}}) \, \, \, \, \text{where} \, \, \, \,   G_{\textsf{v}}=\prod_{i\in \mathbb{Z}} GL(\mathsf{v}_i)
$$
Let us denote by $V_n$ the $\textsf{v}_n$-dimensional fundamental representation of $GL(\textsf{v}_n)$. Then this ring is generated by the classes $\tau_{n,r}:=\bigwedge^r V_n$. Let us define a homomorphism 
$$
T: \textrm{char}(G_{\textsf{v}}) \rightarrow \textsf{RS}
$$
by $T(\tau_{n,r})=T^{n,r}_{\lambda}$. Applying the substitution $z_{i}=\frac{\zeta_{i-1}}{\zeta_i}$, the elements of  $\textsf{RS}$ act as $q$-difference operators on the vertex functions. Our main result involves expressing the insertion of a descendant as the action of such operators:

\begin{Theorem}\label{tRS}
The insertion of descendent $\tau$ into the bare vertex function can be expressed as
\begin{align*}
    T(\tau) \textbf{V}_{\lambda}(\zz) = \textbf{V}_{\lambda}^{(\tau)}(\zz)
\end{align*}
\end{Theorem}
Clearly, to prove the theorem it is enough to show that
$$
T^{n,r}_{\lambda} \textbf{V}_{\lambda}(\zz) = \textbf{V}_{\lambda}^{(\tau_{n,r})}(\zz). 
$$
\subsection{}

Before giving the proof of Theorem \ref{tRS}, we explore an important consequence. Define
\begin{equation*}
   z_{\square} := \prod_{\square' \in H_{\lambda}(\square)} \widehat{z}_{c(\square')}
\end{equation*}
where the shifted parameters $\widehat{z}_i$ are
\begin{equation*}
    \widehat{z}_i:=\left(\frac{\hbar}{q}\right)^{\sigma_{\lambda}(i)} z_i \ \ \ \text{where} \ \ \   \sigma_{\lambda}(i):= \begin{cases} 
      \textsf{v}_{i-1}-\textsf{v}_{i} & \text{if} \ \ i \neq 0 \\
 \textsf{v}_{i-1}-\textsf{v}_{i}+1 & \text{if} \ \ i = 0
   \end{cases} 
\end{equation*}
and $H_{\lambda}(\square)$ denotes the set of boxes in the hook based at $\square$ in $\lambda$. See Figure \ref{yng1}.

\begin{figure}[ht]
\centering
\begin{tikzpicture}[roundnode/.style={circle, draw=black, very thick, minimum size=5mm},squarednode/.style={rectangle, draw=black, very thick, minimum size=5mm}] 
\draw[ultra thick] (-5,5)--(0,0) -- (4,4);
\draw[ultra thick] (-4,6)--(1,1);
\draw[ultra thick] (-5,5)--(-4,6);
\draw[ultra thick] (-2,6)--(2,2);
\draw[ultra thick] (-4,4)--(-2,6);
\draw[ultra thick]  (0,6)--(3,3);
\draw[ultra thick] (2,6)--(4,4);
\draw[ultra thick] (-1,1)--(3,5);
\draw[ultra thick] (-2,2)--(2,6);
\draw[ultra thick] (-3,3)--(0,6);
\draw[ultra thick] (4,4)--(2,6);
\draw[fill=black,opacity=.1] (2,4)--(0,6)--(-1,5)--(2,2)--(4,4)--(3,5)--cycle;
\end{tikzpicture}
\caption{The shaded boxes are an example of a hook in $\lambda=(5,4,3,2)$. If $\square$ is the box at the base of the hook shown, then $z_{\square}= z_0 \left(\frac{\hbar}{q}z_1 \right) z_2 \left(\frac{\hbar}{q}z_3\right)$ and $\zeta_{\square}=\frac{\hbar}{q} \zeta_{3}$.} \label{yng1}
\end{figure}
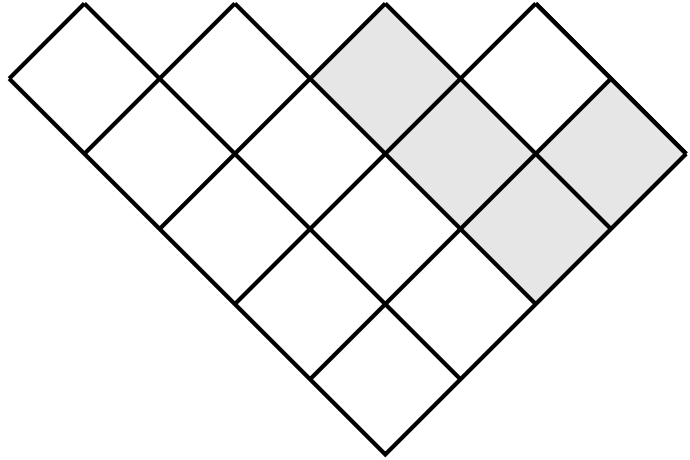

In \cite{dinksmir2}, the authors prove the following formula for $\textbf{V}_{\lambda}(\zz)$:

\begin{Theorem}(\cite{dinksmir2} Theorem 1)
\begin{align} \label{prform}
    \textbf{V}_{\lambda}(\zz)= \prod\limits_{\square \in \lambda} \prod\limits_{i=0}^{\infty} \dfrac{1-\hbar  z_{\square} q^{i}}{1- z_{\square} q^{i}}.
\end{align}
\end{Theorem}

Note that $z_{\square}$ can be expressed in terms of the $\zeta_i$ as
\begin{equation}\label{zzeta}
z_{\square} =  \frac{\zeta_{c_{\lambda}(\square)-l_{\lambda}(\square)-1}}{\zeta_{c_{\lambda}(\square)+a_{\lambda}(\square)}} \prod_{i=c_{\lambda}(\square)-l_{\lambda}(\square)}^{c_{\lambda}(\square)+a_{\lambda}(\square)} \left(\frac{\hbar}{q}\right)^{\sigma_{\lambda}(i)}
\end{equation}

Using our notation above, the set of all boxes with the same content as a given box $\square\in \lambda$ is denoted $C_{\lambda}(c_{\lambda}(\square))$, There is a minimal rectangular Young diagram $\mu\subset \lambda$ containing $C_{\lambda}(c_{\lambda}(\square))$. Explicitly, if $c_{\lambda}(\square)=i<0$, then $\mu=(\mu_1,\ldots,\mu_{\mathsf{v}_{i}})$ where $\mu_j=-i+v_{i}$. If $c_{\lambda}(\square)=i\geq 0$, then $\mu=(\mu_1,\ldots,\mu_{i+\mathsf{v}_i})$ where $\mu_j=\mathsf{v}_i$. With this notation, we define the slice inside $\lambda$ through $\square$ as:
$$
S_{\lambda}(\square):=\begin{cases}
\text{row containing} \,\, \square \subset \mu & \text{if} \  c_{\lambda}(\square)\geq 0 \\
\text{column containing} \,\, \square \subset \mu & \text{if} \ c_{\lambda}(\square)<0
\end{cases}
$$
where the row and column are understood in the ``rectangular" sense, see Figure \ref{yng3}.

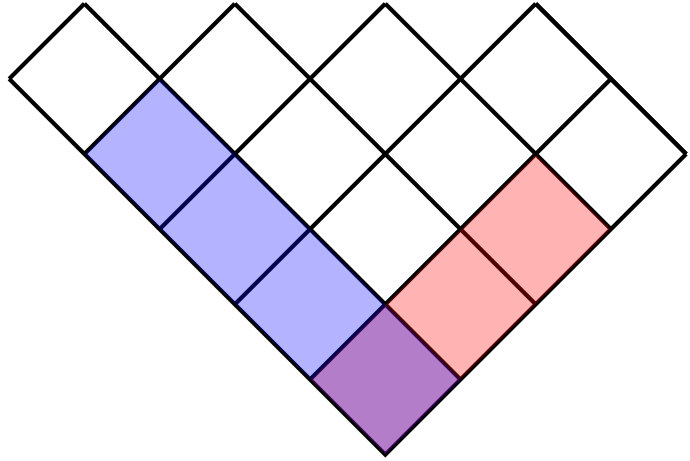
\begin{figure}[ht]
\centering
\begin{tikzpicture}[roundnode/.style={circle, draw=black, very thick, minimum size=5mm},squarednode/.style={rectangle, draw=black, very thick, minimum size=5mm}] 
\draw[ultra thick] (-5,5)--(0,0) -- (4,4);
\draw[ultra thick] (-4,6)--(1,1);
\draw[ultra thick] (-5,5)--(-4,6);
\draw[ultra thick] (-2,6)--(2,2);
\draw[ultra thick] (-4,4)--(-2,6);
\draw[ultra thick]  (0,6)--(3,3);
\draw[ultra thick] (2,6)--(4,4);
\draw[ultra thick] (-1,1)--(3,5);
\draw[ultra thick] (-2,2)--(2,6);
\draw[ultra thick] (-3,3)--(0,6);
\draw[ultra thick] (4,4)--(2,6);
\draw[fill=red,opacity=.3] (0,0)--(3,3)--(2,4)--(-1,1)--cycle;
\draw[fill=blue,opacity=.3] (0,0)--(-4,4)--(-3,5)--(1,1)--cycle;
\end{tikzpicture}
\caption{Two slices inside $\lambda$ are shown. The red boxes make up the slice through the box with content 1 and height 1. The blue boxes make up the slice through the box with content -2 and height 1.} \label{yng3}
\end{figure}

Combining Theorem \ref{tRS} with (\ref{prform}), we obtain the following corollary, giving an explicit combinatorial formula for the capped vertex with descendant.
\begin{Corollary}\label{cor}
\begin{align*}
    \hat{\textbf{V}}_{\lambda}^{(\tau_{n,r})}(\zz)= \hbar^{r(r-1)/2+\beta(n)} \sum_{\substack{I \subset C_{\lambda}(n) \\ |I|=r }} \prod_{\substack{\square \in I \\ \square' \notin I}} \frac{\hbar \zeta_{\square'} -  \zeta_{\square}}{\zeta_{\square'} - \zeta_{\square}} \prod_{\square\in I}\prod_{\square' \in S_{\lambda}(\square)} \frac{1-z_{\square'}}{1-\hbar z_{\square'}}
\end{align*}
\end{Corollary}
\begin{proof}
By Theorem \ref{capping}, we know that $$ \hat{\textbf{V}}_{\lambda}^{(\tau_{n,r})}(\zz)= \frac{ \textbf{V}_{\lambda}^{(\tau_{n,r})}(\zz)}{ \textbf{V}_{\lambda}(\zz)}$$
Applying Theorem \ref{tRS}, we obtain
$$ \hat{\textbf{V}}_{\lambda}^{(\tau_{n,r})}(\zz)=\frac{T_{\lambda}^{n,r} \textbf{V}_{\lambda}(\zz)}{ \textbf{V}_{\lambda}(\zz)}$$
From (\ref{pbox}) and (\ref{zzeta}), we see that the operator $p_{\square}$ has the following effect on $z_{\square'}$:
$$p_{\square} \left(z_{\square'}\right) =
\begin{cases}
qz_{\square'}, & \square' \in S_{\lambda}(\square) \\
z_{\square'}, & \square' \notin S_{\lambda}(\square)
\end{cases}
$$
From (\ref{prform}), we see that the bare vertex function transforms under the scaling of a fixed $z_{\square}$ by $q$ as follows:
\begin{equation*}
    \textbf{V}_{\lambda}(\zz)\big|_{z_{\square}=qz_{\square}} = \frac{1-z_{\square}}{1-\hbar z_{\square}} \textbf{V}_{\lambda}(\zz)
\end{equation*}
Putting all this together, we have 
\begin{align*}
  &\hat{\textbf{V}}^{(\tau_{n,r})}(\zz) =  \frac{  T_{\lambda}^{n,r} \textbf{V}_{\lambda}(\zz)}{\textbf{V}_{\lambda}(\zz)} \\
  &= \frac{1}{\textbf{V}_{\lambda}(\zz)} \left(\hbar^{r(r-1)/2+\beta(n)} \sum_{\substack{I \subset C_{\lambda}(n) \\ |I|=r }} \prod_{\substack{\square \in I \\ \square' \notin I}} \frac{\hbar \zeta_{\square'} -  \zeta_{\square}}{\zeta_{\square'} - \zeta_{\square}} \prod_{\square\in I}\prod_{\square' \in S_{\lambda}(\square)} \frac{1-z_{\square'}}{1-\hbar z_{\square'}} \right) \textbf{V}_{\lambda}(\zz) \\
  &= \hbar^{r(r-1)/2+\beta(n)} \sum_{\substack{I \subset C_{\lambda}(n) \\ |I|=r }} \prod_{\substack{\square \in I \\ \square' \notin I}} \frac{\hbar \zeta_{\square'} -  \zeta_{\square}}{\zeta_{\square'} - \zeta_{\square}} \prod_{\square\in I}\prod_{\square' \in S_{\lambda}(\square)} \frac{1-z_{\square'}}{1-\hbar z_{\square'}}
\end{align*} 
\end{proof}
We note that for $|q|<1$, $\hat{\textbf{V}}_{\lambda}^{(\tau_{n,r})}(\zz)$ converges to a rational function of the K\"ahler parameters. Rationality of descendant insertions in the case of the cohomology of the Hilbert scheme of points in $\mathbb{C}^2$ was established in \cite{PPRationality}. It is expected that the same is true in $K$-theory for general Nakajima quiver varieties, see \cite{OkBethe}.

\section{Proof of Theorem \ref{tRS}}

\subsection{}
The vertex functions for type $A$ Nakajima quiver varieties can be described by natural integral representations of Mellin-Barnes type see Section 1.1.5-1.1.6 in \cite{OkBethe}. These integral representations were investigated for the quiver varieties isomorphic to cotangent bundles over Grassmannians and partial flag varieties in \cite{Pushk1,Pushk2,tRSKor}. We explain below what this description looks like.

Let $X$ be a Nakajima quiver variety arising from a type $A_n$ quiver with vertex set $I$, dimension vector $\textsf{v}$, and framing dimension vector $\textsf{w}$.

For a character $w_1+ \ldots + w_m \in K_{\bT}(X)$ we denote
\begin{equation*}\Phi(w_1+\ldots + w_m)=\varphi(w_1)\dots \varphi(w_m), \ \ \ \ \ \varphi(w):=\prod_{n=0}^{\infty}\left( 1-w q^n \right).
\end{equation*}
and extend this definition by linearity to polynomials in $K_\bT(X)$ with negative coefficients. Let $\mathcal{P}$ be the bundle over $X$ associated to the virtual $G$-module 
\be
\bigoplus_{i\in I} Hom(W_i,V_i) + \bigoplus_{i\to j} Hom(V_i,V_j) - \bigoplus_{i\in I} Hom(V_i,V_i)
\ee
where $i\to j$ denotes the sum over the arrows of the quiver. 

If $x_{i,1},\dots,x_{i,\mathsf{v}_i}$ denote the Grothendieck roots of $i$-th tautological bundle (i.e. the bundle over $X$ associated to the $G$-module $V_i$) and $a_{i,j}$ denote the equivariant parameters associated to the framings for $j=1,\ldots, \mathsf{w}_i$, then
\begin{align*}
    \mathcal{P}= \sum_{i \in I} \left( \sum_{j=1}^{\mathsf{v}_i} x_{i,j}^{-1} \right) \left( \sum_{j=1}^{\mathsf{v}_{i+1}} x_{i+1,j} \right) +\sum_{i \in I} \left( \sum_{j=1}^{\textsf{w}_i} a_{i,j}^{-1} \right) \left( \sum_{j=1}^{\mathsf{v}_i} x_{i,j} \right) \\
    - \sum_{i\in I} \left( \sum_{j=1}^{\mathsf{v}_i} x_{i,j}^{-1} \right) \left( \sum_{j=1}^{\mathsf{v}_i} x_{i,j} \right) \in K_\bT(X).
\end{align*}

We abbreviate the set of Grothendieck roots of the tautological bundles by $\bs{x}$ and define the following formal expression:
$$
{\bf e}({\bs x},\zz):=\exp\Big(\dfrac{1}{\ln{q}} \sum\limits_{ i \in I } \ln(z_i)  \ln(\det \mathcal{V}_i) \Big)=\exp\Big(\dfrac{1}{\ln{q}} \sum\limits_{ i \in I }\sum_{j=1}^{\textsf{v}_i} \ln(z_i)  \ln(x_{i,j}) \Big)
$$
For $t\in K_{\bT}(X)$ and $p\in X^{\bT}$, let $t|_p \in K_{\bT}(p)$ be the restriction of $t$. Then the restriction of the vertex function to a fixed point $p\in X^{\bT}$ is
\begin{equation} \label{verpow}
{\bf V}^{(\tau)}_p(\aa,\zz)= \frac{1}{\Phi\Big((q-\hbar) \mathcal{P}|_p \Big) {\bf e}({\bs x}|_p,\zz)} {\bf \tilde V}^{(\tau)}_p(\aa,\zz),
\end{equation}
where
\begin{equation*}\label{vertild}
\tilde{{\bf V}}^{(\tau)}_p(\aa,\zz)= \prod\limits_{i,j} \int\limits_{0}^{x_{i,j}|_p}  d_q x_{i,j} \, \Phi\Big((q-\hbar) \mathcal{P} \Big) {\bf e}({\bs x},\zz)\tau({\bs{x}})
\end{equation*}
and the integral symbols denote the Jackson $q$-integral over all Grothendieck roots:
$$
\int\limits_{0}^{a}  d_q x f(x) :=\sum\limits_{n=0}^{\infty}\, f(a q^n),
$$
For an indeterminate $x$, we define the $q$-Pochhammer symbol by
\begin{equation*}
(x)_d:=\frac{\varphi(x)}{\varphi(xq^d)} 
\end{equation*}
We note that (\ref{verpow}) is a power series in $z_i$ with coefficients given by combinations of $q$-Pochhammer symbols in the equivariant parameters. 

Equation (\ref{verpow}) arises when one analyzes the torus fixed points on the quasimap moduli space and computes the vertex function using $K$-theoretic equivariant localization.

\subsection{}
The cotangent bundle of the full flag variety can be described as a Nakajima quiver variety corresponding to the $A_{n-1}$ quiver with dimension vectors $\mathsf{v}=(n-1,n-2, \ldots,0)$ and $\mathsf{w}=(n,0,\ldots,0)$. In \cite{tRSKor}, Koroteev proves that the vertex function of the cotangent bundle of the full flag variety restricted to an appropriate fixed point is an eigenvector of the tRS operators, with eigenvalues given by the elementary symmetric functions of the equivariant parameters. 

More generally, we can allow some redundant information which disappears upon taking the quotient. We start with a partition $\lambda$ with associated dimension $\mathsf{v}$ as before, and some integer $n$ which will correspond to the content of a set of boxes in $\lambda$. If $n\geq 0$, we consider the quiver variety with dimensions given by the number of boxes in the Young diagram strictly to the right of the $n$th column. If $n<0$, we obtain our dimensions from the portion of the partition strictly to the left of the $n$th column. The number of boxes with content $n$ gives the framing dimension.

The corresponding quiver variety is canonically isomorphic to the cotangent bundle of the variety parameterizing full flags inside a $\mathsf{v}_n$-dimensional vector space. The vertex functions of the quiver varieties differ, as the vertex functions are sensitive to the way in which the quotient is taken. Because of this difference, we refer to the flag variety obtained by a partition $\lambda$ and choice of $n$ as a ``redundant flag variety."

For an example of this ``partition truncation", see Figure \ref{yng4}. In this notation, the usual quiver variety description of the cotangent bundle of the full flag variety can be obtained from square partitions.

\begin{figure}[ht]
\centering
\begin{tikzpicture}[roundnode/.style={circle, draw=black, very thick, minimum size=5mm},squarednode/.style={rectangle, draw=black, very thick, minimum size=5mm}] 
\draw[ultra thick, dashed] (-5,5)--(-1,1);
\draw[ultra thick] (-1,1)--(0,0) -- (5,5);
\draw[ultra thick] (4,6)--(5,5);
\draw[ultra thick,dashed] (-4,6)--(-1,3);
\draw[ultra thick] (-1,3)--(1,1);
\draw[ultra thick,dashed] (-5,5)--(-4,6);
\draw[ultra thick,dashed] (-2,6)--(-1,5);
\draw[ultra thick] (-1,5)--(2,2);
\draw[ultra thick,dashed] (-4,4)--(-2,6);
\draw[ultra thick]  (0,6)--(3,3);
\draw[ultra thick] (2,6)--(4,4);
\draw[ultra thick] (-1,1)--(4,6);
\draw[ultra thick] (-1,3)--(2,6);
\draw[ultra thick,dashed] (-2,2)--(-1,3);
\draw[ultra thick,dashed] (-3,3)--(-1,5);
\draw[ultra thick] (-1,5)--(0,6);
\draw[ultra thick] (4,4)--(2,6);
\node[squarednode](F) at (0,-1){3};
\node[roundnode](V1) at (1,-1){2};
\node[roundnode](V2) at (2,-1){2};
\node[roundnode](V3) at (3,-1){1};
\node[roundnode](V4) at (4,-1){1};
\draw[very thick, ->] (F) -- (V1);
\draw[very thick, ->] (V1) -- (V2);
\draw[very thick, ->] (V2) -- (V3);
\draw[very thick, ->] (V3) -- (V4);
\end{tikzpicture}
\caption{Obtaining the dimension information for a redundant flag variety from the partition $(5,4,3,2,1)$ with $n=0$. The framing corresponds to the square node.} \label{yng4}
\end{figure}
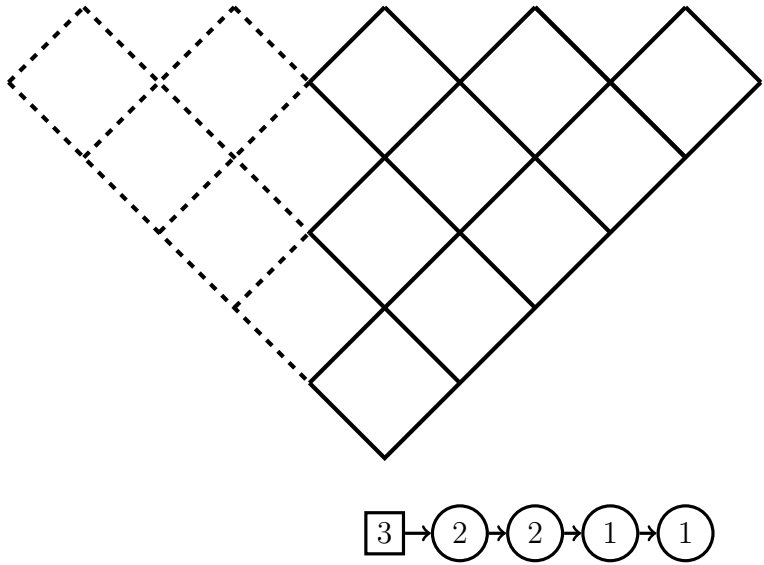

With this in mind, we can make sense of $\zeta_{\square}$ as before for $\square$ inside a truncated partition.

Let 
\begin{align*}
    T_r=\hbar^{r(r-1)/2} \sum_{\substack{I \subset C_{\lambda}(n) \\ |I|=r }} \prod_{\substack{\square \in I \\ \square' \notin I}} \frac{\hbar \zeta_{\square'} - \zeta_{\square}}{\zeta_{\square'} - \zeta_{\square}} \prod_{\square \in I} \widetilde{p}_{\square}
\end{align*}
where $\widetilde{p}_{\square}=\frac{a_{n-h_{\lambda}(\square)}}{\hbar^{n-h_{\lambda}(\square)-1}} p_{\square}$. 

As before, with the change of variables $z_{i}=\frac{\zeta_{i-1}}{\zeta_i}$, the operators $T_r$ act on the vertex functions of the redundant flag variety. In our notation, Koroteev's theorem can be generalized to the following:

\begin{Theorem}(\cite{tRSKor} Theorem 2.6)\label{Kor}
Fix $\lambda$ and $n$. Let ${\bf V}_p(\aa,\zz)$ be the bare vertex function of the associated redundant flag variety restricted to the fixed point $p$ at which the weights of the tautological bundle $\mathcal{V}_i$ are $\{a_{j} : 1\leq j \leq \mathsf{v}_i\}$ for all $i$. Then
\begin{align*}
   T_r {\bf V}_p(\aa,\zz)= e_r(\aa) {\bf V}_p(\aa,\zz)
\end{align*}
where $e_r(\aa)$ denotes the $r$th elementary symmetric function in the equivariant parameters.
\end{Theorem}

\subsection{}
Let $\lambda$ be a partition so that the associated dimension vector is $\mathsf{v}=(\mathsf{v}_{-r}, \ldots, \mathsf{v}_s)$. Fix $n\in \mathbb{Z}$. For definiteness, we assume $n\geq 0$. Let $\textbf{V}_p(\aa, \zz)$ be as in Theorem \ref{Kor}. In other words, $\textbf{V}_p(\aa, \zz)$ is the restriction to a fixed point of the vertex function of the quiver variety with dimension $(\mathsf{v}_{n+1}, \ldots, \mathsf{v}_s)$ and framing dimension $(\mathsf{v}_n, 0, \ldots, 0)$.

\begin{Note}
The proof of Theorem \ref{tRS} involves interpreting certain terms in the power series for $\textbf{V}_{\lambda}(\zz)$ as vertex functions for redundant flag varieties with specialized equivariant parameters. For definiteness, we will assume throughout that $n\geq 0$. If $n<0$, the same arguments given below can be modified in the obvious manner.
\end{Note}

\begin{Lemma}\label{lem1} Specializing the equivariant parameters to $a_i=\hbar^{i-1} q^{d_{n,i}}$ and relabeling the K\"ahler paramters in $\textbf{V}_p(\aa,\zz)$, we have
\begin{align*}
  \textbf{V}_{\lambda}(\zz)= \sum_{\substack{d_{i,j}}} \Psi \prod_{i=-r}^{n} \prod_{j=1}^{\mathsf{v}_i} z_{i}^{d_{i,j}} \prod_{i=n+1}^{s} \prod_{j=1}^{\mathsf{v}_i} z_{i}^{d_{n,j}}  {\bf V}_p(\aa,z_{n+1},\ldots,z_{s}) \big|_{a_i=\hbar^{i-1} q^{d_{n,i}}} 
\end{align*}
where $\Psi$ represents $\varphi$ function terms in (\ref{verpow}) that do not depend on $d_{i,j}$ for $i>n$. 
\end{Lemma}


\begin{proof}

From (\ref{verpow}) vertex function for $X_{\lambda}$ is 
\begin{align}\nonumber
  \textbf{V}_{\lambda}(\zz)  = \sum_{d_{i,j}} \prod_{j=1}^{\mathsf{v}_0}  \frac{\varphi(\hbar^j)}{\varphi(\hbar^j q^{d_{0,j}})} \frac{\varphi(q\hbar^{j-1}q^{d_{0,j}})}{\varphi(q\hbar^{j-1})}  \\ \nonumber
  \prod_{i<0} \prod_{j=1}^{\mathsf{v}_i} \prod_{k=1}^{\mathsf{v}_{i+1}}  \frac{\varphi(\hbar^{k-j})}{\varphi(\hbar^{k-j} q^{d_{i+1,k}-d_{i,j}})} \frac{\varphi(q \hbar^{k-j-1} q^{d_{i+1,k}-d_{i,j}})}{\varphi(q \hbar^{k-j-1} )}  \\  \nonumber
    \prod_{i\geq 0} \prod_{j=1}^{\mathsf{v}_i} \prod_{k=1}^{\mathsf{v}_{i+1}}   \frac{\varphi(\hbar^{k-j+1})}{\varphi(\hbar^{k-j+1} q^{d_{i+1,k}-d_{i,j}})} \frac{\varphi(q\hbar^{k-j} q^{d_{i+1,k}-d_{i,j}})}{\varphi(q\hbar^{k-j})}  \\ \label{vf1}
    \prod_{i\in \mathbb{Z}} \prod_{j,k=1}^{\mathsf{v}_i} \frac{\varphi(q\hbar^{k-j})}{\varphi(q\hbar^{k-j} q^{d_{i,k}-d_{i,j}})} \frac{\varphi(\hbar^{k-j+1} q^{d_{i,k}-d_{i,j}})}{\varphi(\hbar^{k-j+1})} \, \zz^\dd
\end{align}
where each $d_{i,j}$ is summed from $0$ to $\infty$.

Separating the terms corresponding to the $i$th column for $i>n$, we have
\begin{align*}
    \textbf{V}_{\lambda}(\zz) &= \sum_{d_{i,j}} \Psi \prod_{i=-r}^{n} \prod_{j=1}^{\mathsf{v}_i} z_{i}^{d_{i,j}} \prod_{j=1}^{\mathsf{v}_n} \prod_{k=1}^{\mathsf{v}_{n+1}} \frac{\varphi\left(q \hbar^{i-j} q^{d_{n+1,k}-d_{n,j}} \right) }{\varphi\left(\hbar \hbar^{i-j} q^{d_{n+1,k}-d_{n,j}} \right)}
    \\ & \prod_{i=n+1}^{s-1} \prod_{j=1}^{\mathsf{v}_i} \prod_{k=1}^{\mathsf{v}_{i+1}} \frac{\varphi\left(q \hbar^{k-j} q^{d_{i+1,k}-d_{i,j}}\right) }{ \varphi\left(\hbar \hbar^{k-j} q^{d_{i+1,k}-d_{i,j}}\right)} \\ & \prod_{i=n+1}^{s} \prod_{j,k=1}^{\mathsf{v}_i} \frac{\varphi\left(\hbar \hbar^{k-j} q^{d_{i,k}-d_{i,j}}\right) }{\varphi\left(q \hbar^{k-j} q^{d_{i,k}-d_{i,j}}\right)} \prod_{i=n+1}^{s} \prod_{j=1}^{\mathsf{v}_i} z_i^{d_{i,j}}
\end{align*}
where $\Psi$ represents some $\varphi$ function terms, which do not depend on $d_{i,j}$ for $i> n$.

On the other hand, from (\ref{verpow}) the vertex function of the redundant flag variety arising from a partition $\lambda$ and integer $n$ at the fixed point $p$ has the form

\begin{align*}
    {\bf V}_p(\aa, z_{n+1}, \ldots, z_s) &= \sum_{f_{i,j}} \prod_{j=1}^{\mathsf{v}_n} \prod_{k=1}^{\mathsf{v}_{n}+1} \frac{\varphi\left(q\frac{a_{k}}{a_{j}}q^{f_{n+1,k}}\right)}{\varphi\left(q\frac{a_{k}}{a_{j}}\right)} \frac{\varphi\left(\hbar\frac{a_{k}}{a_{j}}\right)}{\varphi\left(\hbar\frac{a_{k}}{a_{j}}q^{f_{n+1,k}}\right)}  \\
    & \prod_{i=n+1}^{s-1} \prod_{j=1}^{\mathsf{v}_i} \prod_{k=1}^{\mathsf{v}_{i+1}} \frac{\varphi\left(q\frac{a_{k}}{a_{j}}q^{f_{i+1,k}-f_{i,j}}\right)}{\varphi\left(q\frac{a_{k}}{a_{j}}\right)} \frac{\varphi\left(\hbar\frac{a_{k}}{a_{j}}\right)}{\varphi\left(\hbar \frac{a_{k}}{a_{j}}q^{f_{i+1,k}-f_{i,j}}\right)} \\
    & \prod_{i=n+1}^{s} \prod_{j,k=1}^{\mathsf{v}_i} \frac{\varphi\left(\hbar\frac{a_{k}}{a_{j}}q^{f_{i,k}-f_{i,j}}\right)}{\varphi\left(\hbar \frac{a_k}{a_j}\right)} \frac{\varphi\left(q \frac{a_k}{a_j}\right)}{\varphi\left(q\frac{a_{k}}{a_{j}}q^{f_{i,k}-f_{i,j}}\right)} \prod_{i=n+1}^{s} \prod_{j=1}^{\mathsf{v}_i} z_i^{f_{i,j}}
\end{align*}

We reindex the summation as follows. First, replace $f_{n,j}$ by $d_{n,j}$. Second, substitute $f_{i,j}$ by $d_{i,j}-d_{n,j}$ for $i>n$. By examining the terms in $\textbf{V}_{\lambda}(\zz)$, we see that a term in the sum is only nonzero if and only if the $d_{i,j}$ give a set of interlacing partitions (see \cite{dinksmir2} Proposition 7). In particular, we must have $d_{i,j}-d_{n,j}\geq 0$ and so the sum for $\textbf{V}_p(\aa,z_{n+1},\ldots, z_s)$ can be reindexed as 

\begin{align*}
      &{\bf V}_p(\aa, z_{n+1}, \ldots, z_s) = \sum_{d_{i,j}} \prod_{j=1}^{\mathsf{v}_n} \prod_{k=1}^{\mathsf{v}_{n}+1} \frac{\varphi\left(q\frac{a_{k}}{a_{j}}q^{d_{n+1,k}-d_{n,k}}\right)}{\varphi\left(\hbar\frac{a_{k}}{a_{j}}q^{d_{n+1,k}-d_{n,k}}\right)} \frac{\varphi\left(\hbar\frac{a_{k}}{a_{j}}\right)}{\varphi\left(q\frac{a_{k}}{a_{j}}\right)}  \\
    & \prod_{i=n+1}^{s-1} \prod_{j=1}^{\mathsf{v}_i} \prod_{k=1}^{\mathsf{v}_{i+1}} \frac{\varphi\left(q\frac{a_{k}}{a_{j}}q^{d_{i+1,k}-d_{n,k}-(d_{i,j}-d_{n,j})}\right)}{\varphi\left(\hbar \frac{a_{k}}{a_{j}}q^{d_{i+1,k}-d_{n,k}-(d_{i,j}-d_{n,j})}\right)} \frac{\varphi\left(\hbar\frac{a_{k}}{a_{j}}\right)}{\varphi\left(q\frac{a_{k}}{a_{j}}\right)} \\
    & \prod_{i=n+1}^{s} \prod_{j,k=1}^{\mathsf{v}_i} \frac{\varphi\left(\hbar\frac{a_{k}}{a_{j}}q^{d_{i,k}-d_{n,k}-(d_{i,j}-d_{n,j})}\right)}{\varphi\left(q\frac{a_{k}}{a_{j}}q^{d_{i,k}-d_{n,k}-(d_{i,j}-d_{n,j})}\right)} \frac{\varphi\left(q \frac{a_k}{a_j}\right)}{\varphi\left(\hbar \frac{a_k}{a_j}\right)} \prod_{i=n+1}^{s} \prod_{j=1}^{\mathsf{v}_i} z_i^{d_{i,j}-d_{n,j}}
\end{align*}

Next, we substitute $a_i=\hbar^{i-1}q^{d_{n,i}}$ and obtain

\begin{align*}
     & {\bf V}_p(\aa, z_{n+1}, \ldots, z_s)\big|_{a_i=\hbar^{i-1}q^{d_{n,i}}} \\
     &= \sum_{d_{i,j}} \prod_{j=1}^{\mathsf{v}_n} \prod_{k=1}^{\mathsf{v}_{n}+1} \frac{\varphi\left(q\hbar^{k-j} q^{d_{n+1,k}-d_{n,j}}\right)}{\varphi\left(\hbar \hbar^{k-j} q^{d_{n+1,k}-d_{n,j}}\right)} \frac{\varphi\left(\hbar\hbar^{k-j} q^{d_{n,k}-d_{n,j}}\right)}{\varphi\left(q\hbar^{k-j} q^{d_{n,k}-d_{n,j}}\right)}  \\
    & \prod_{i=n+1}^{s-1} \prod_{j=1}^{\mathsf{v}_i} \prod_{k=1}^{\mathsf{v}_{i+1}} \frac{\varphi\left(q \hbar^{k-j} q^{d_{i+1,k}-d_{i,j}}\right)}{\varphi\left(\hbar \hbar^{k-j} q^{d_{i+1,k}-d_{i,j}}\right)} \frac{\varphi\left(\hbar\hbar^{k-j}q^{d_{n,k}-d_{n,j}}\right)}{\varphi\left(q\hbar^{k-j}q^{d_{n,k}-d_{n,j}}\right)} \\
    & \prod_{i=n+1}^{s} \prod_{j,k=1}^{\mathsf{v}_i} \frac{\varphi\left(\hbar\hbar^{k-j}q^{d_{i,k}-d_{i,j}}\right)}{\varphi\left(q\hbar^{k-j}q^{d_{i,k}-d_{i,j}}\right)} \frac{\varphi\left(q \hbar^{k-j} q^{d_{n,k}-d_{n,j}}\right)}{\varphi\left(\hbar \hbar^{k-j} q^{d_{n,k}-d_{n,j}}\right)} \prod_{i=n+1}^{s} \prod_{j=1}^{\mathsf{v}_i} z_i^{d_{i,j}-d_{n,j}}
\end{align*}
We simplify some of the terms in the above summation:
\begin{align*}
      & \prod_{j=1}^{\mathsf{v}_n} \prod_{k=1}^{\mathsf{v}_{n}+1}  \frac{\varphi\left(\hbar\hbar^{k-j} q^{d_{n,k}-d_{n,j}}\right)}{\varphi\left(q\hbar^{k-j} q^{d_{n,k}-d_{n,j}}\right)}  \prod_{i=n+1}^{s-1} \prod_{j=1}^{\mathsf{v}_i} \prod_{k=1}^{\mathsf{v}_{i+1}}  \frac{\varphi\left(\hbar\hbar^{k-j}q^{d_{n,k}-d_{n,j}}\right)}{\varphi\left(q\hbar^{k-j}q^{d_{n,k}-d_{n,j}}\right)} \\ &\prod_{i=n+1}^{s} \prod_{j,k=1}^{\mathsf{v}_i}  \frac{\varphi\left(q \hbar^{k-j} q^{d_{n,k}-d_{n,j}}\right)}{\varphi\left(\hbar \hbar^{k-j} q^{d_{n,k}-d_{n,j}}\right)} \\
      &=  \prod_{i=n}^{s-1} \prod_{j=1}^{\mathsf{v}_i} \prod_{k=1}^{\mathsf{v}_{i+1}}  \frac{\varphi\left(\hbar\hbar^{k-j}q^{d_{n,k}-d_{n,j}}\right)}{\varphi\left(q\hbar^{k-j}q^{d_{n,k}-d_{n,j}}\right)} \prod_{i=n+1}^{s} \prod_{j,k=1}^{\mathsf{v}_i}  \frac{\varphi\left(q \hbar^{k-j} q^{d_{n,k}-d_{n,j}}\right)}{\varphi\left(\hbar \hbar^{k-j} q^{d_{n,k}-d_{n,j}}\right)} \\
      &= \prod_{i=n}^{s-1} \prod_{j=\mathsf{v}_{i+1}+1}^{\mathsf{v}_i} \prod_{k=1}^{\mathsf{v}_{i+1}} \frac{\varphi\left(\hbar\hbar^{k-j}q^{d_{n,k}-d_{n,j}}\right)}{\varphi\left(q\hbar^{k-j}q^{d_{n,k}-d_{n,j}}\right)} \\
      &=\prod_{\substack{j \\ \mathsf{v}_j-\mathsf{v}_{j+1}=1}} \prod_{k=1}^{\mathsf{v}_{j+1}} \frac{\varphi\left(\hbar\hbar^{k-j}q^{d_{n,k}-d_{n,j}}\right)}{\varphi\left(q\hbar^{k-j}q^{d_{n,k}-d_{n,j}}\right)} 
\end{align*}
From (\ref{verpow}), it is easy to see that all of these $\varphi$ function terms appear in $\Psi$. Thus we have shown that 
\begin{align*}
     \textbf{V}_{\lambda}(\zz) &= \sum_{d_{i,j}} \Psi'  \prod_{i=-r}^{n} \prod_{j=1}^{\mathsf{v}_i} z_{i}^{d_{i,j}} \prod_{i=n+1}^{s} \prod_{j=1}^{\mathsf{v}_i} z_i^{d_{n,j}}  {\bf V}_p(\aa, z_{n+1}, \ldots, z_s)\big|_{a_i=\hbar^{i-1}q^{d_{n,i}}}
\end{align*}
where $\Psi'$ is a product of $\varphi$ function terms, none of which depend on $d_{i,j}$ for $i>n$. This proves the lemma.

\end{proof}

\subsection{Conclusion of the proof}
Let $\square \in \lambda$ be the box with content $n$ and height $m$. Substituting $z_{i}=\frac{\zeta_{i-1}}{\zeta_{i}}$ and applying $p_{\square}$ to Lemma \ref{lem1} we have
\begin{align*}
p_{\square} \textbf{V}_{\lambda}(\zz) &=  \sum_{\substack{d_{i,j}}} \Psi \prod_{i=-r}^{n} \prod_{j=1}^{\mathsf{v}_i} \left(\frac{\zeta_{i-1}}{\zeta_{i}} \right)^{d_{i,j}} \prod_{i=n+1}^{s} \prod_{j=1}^{\mathsf{v}_i} \left(\frac{\zeta_{i-1}}{\zeta_{i}} \right)^{d_{n,j}}  q^{d_{n,m}} p_{\square} {\bf V}_p(\aa,\zeta) \big|_{a_i=\hbar^{i-1} q^{d_{n,i}}}   \\
&= \sum_{\substack{d_{i,j}}} \Psi  \prod_{i=-r}^{n} \prod_{j=1}^{\mathsf{v}_i} \left(\frac{\zeta_{i-1}}{\zeta_{i}} \right)^{d_{i,j}} \prod_{i=n+1}^{s} \prod_{j=1}^{\mathsf{v}_i} \left(\frac{\zeta_{i-1}}{\zeta_{i}} \right)^{d_{n,j}}  \left( \widetilde{p}_{\square} {\bf V}_p(\aa,\zeta) \right) \big|_{a_i=\hbar^{i-1} q^{d_{n,i}}}  
\end{align*}
And so 
\begin{align*}
&T_{\lambda}^{n,r} \textbf{V}_{\lambda}(\zz) =\sum_{\substack{d_{i,j}}} \Psi \prod_{i=-r}^{n} \prod_{j=1}^{\mathsf{v}_i} \left(\frac{\zeta_{i-1}}{\zeta_{i}} \right)^{d_{i,j}} \prod_{i=n+1}^{s} \prod_{j=1}^{\mathsf{v}_i} \left(\frac{\zeta_{i-1}}{\zeta_{i}} \right)^{d_{n,j}}  T_r {\bf V}_p(\aa,\zeta) \big|_{a_i=\hbar^{i-1} q^{d_{n,i}}}   \\
&=\sum_{\substack{d_{i,j}}} \Psi  \prod_{i=-r}^{n} \prod_{j=1}^{\mathsf{v}_i} \left(\frac{\zeta_{i-1}}{\zeta_{i}} \right)^{d_{i,j}} \\  &\qquad \qquad \prod_{i=n+1}^{s} \prod_{j=1}^{\mathsf{v}_i} \left(\frac{\zeta_{i-1}}{\zeta_{i}} \right)^{d_{n,j}}  e_r(q^{d_{n,1}}, \ldots, \hbar^{\mathsf{v}_n-1} q^{d_{n,\mathsf{v}_n}}) {\bf V}_p(\aa,\zeta) \big|_{a_i=\hbar^{i-1} q^{d_{n,i}}}   \\
&= \textbf{V}^{(\tau_{n,r})}_{\lambda}(\zz)
\end{align*}
which concludes the proof.

\section{Monodromy of vertex functions}

\subsection{}
Let $\bK$ be a torus with coordinates $\zz=(z_1,\cdots,z_n)$. Let $\sigma=(\sigma_1,\dots,\sigma_n) \in \textrm{cochar}_{\matR}(\bK)$ be a cocharacter. We denote $\zz q^{\sigma}=(z_1 q^{\sigma_1},\dots,z_n q^{\sigma_n})$ for some $q\in \matC^{\times}$ with $|q|<1$. 

Let us consider scalar $q$-difference equations (qde) of the form:
\be \label{qde}
\Psi(\zz q^{\sigma}) = M_{\sigma}(\zz) \Psi(\zz), 
\ee
where $\Psi(\zz)$ denotes a $\matC$-valued function on $\bK$ and $M_{\sigma}(\zz)\in \matQ(z_1,\dots, z_n)$ satisfies
\be \label{limsm}
\lim_{q\to 0} M_{\sigma}(q^{\delta}) \neq  0,\infty
\ee
for generic $\delta \in \textrm{cochar}_{\matR}(\bK)$. 
Clearly, the limits above will not change if we scale the cocharacter $\delta\to l \delta$ for some real $l>0$. This provides a decomposition  $$
\textrm{Lie}_{\matR}(\bK) \supset \coprod\, \fC$$ 
into a set of cones for which the limits (\ref{limsm}) remain the same. The cones $\fC$ are sometimes called {\it asymptotic zones} of the qde (\ref{qde}).

\begin{Definition}
We say that a function $F(\zz)$ is analytic in an asymptotic zone $\fC$ if it is given by a power series
\be \label{fholm}
F(z)=\sum\limits_{\langle \dd, \fC \rangle >0} c_{\dd} \zz^{\dd}  
\ee
with non-zero radius of convergence. Here, $\langle \cdot, \cdot \rangle$ denotes the natural pairing on characters and cocharacters.
\end{Definition}

It is convenient to view the asymptotic zones $\fC$ as ``infinities'' in certain toric compactification $ \overline{\bK}$ of $\bK$. The closure of each chamber $\fC$ is a strongly convex rational polyhedral cone, and the set of such cones generates a fan $\Delta$. The toric variety $ \overline{\bK}$ associated to this fan contains $\bK$ as a subvariety. See Figure \ref{toric} below for an example. The chambers $\fC$ then correspond to  certain points $0_{\fC} \in \overline{\bK}$.  A function $F(z)$ is analytic in an asymptotic zone $\fC$ when (\ref{fholm}) is the Taylor series of a function on  $\overline{\bK}$ holomorphic in a non-zero~neighborhood of~$0_{\fC}$.

\begin{figure}[h]
\centering
\includegraphics[width=10cm]{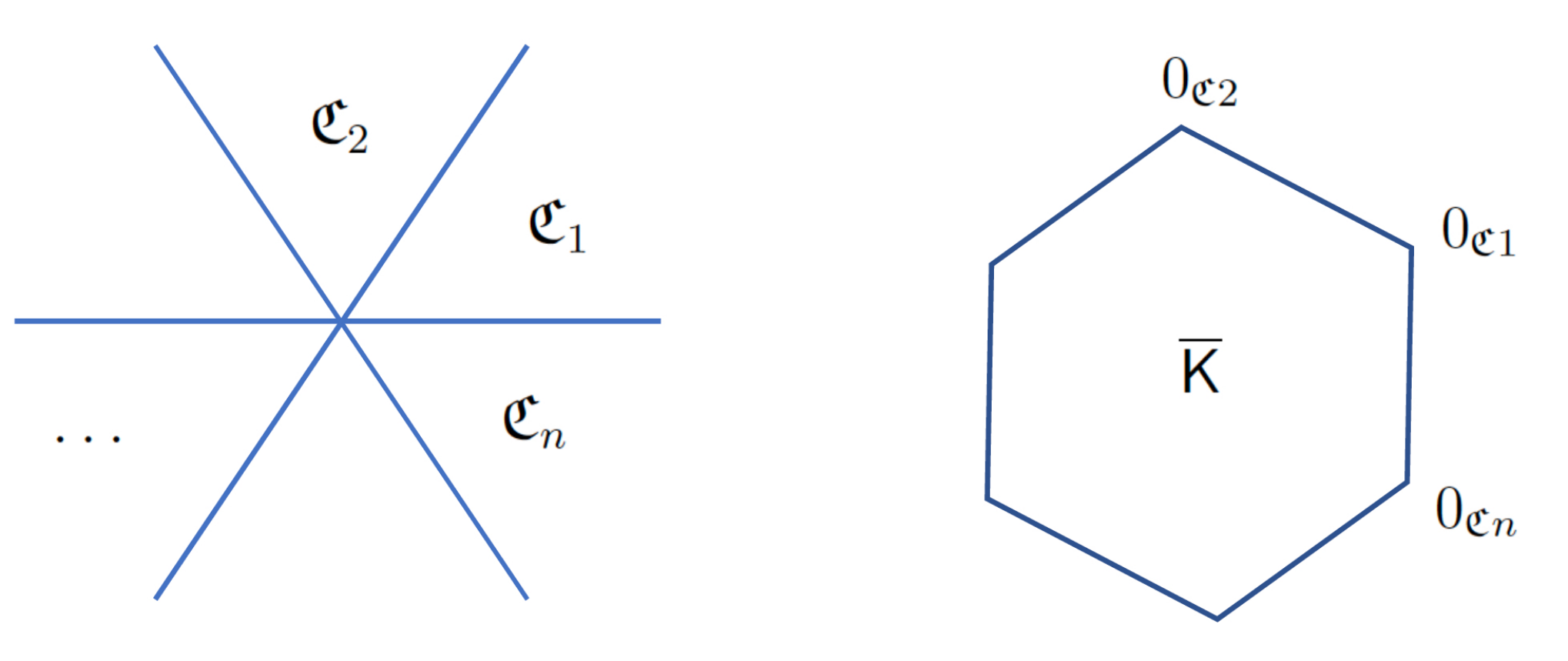}
\caption{\label{toric} An arrangement of chambers in $\Lie_{\matR}(\bK)$ and the corresponding toric compactification $\overline{\bK}$. }
\end{figure}

We will use the following notation
$$
M_{\sigma}(0_{\fC}):=\lim\limits_{q\to 0} M_{\sigma}(q^{\delta}), \ \ \delta \in \fC. 
$$

\begin{Proposition}
For a chamber $\fC$ there exists a unique solution of (\ref{qde}) of the form:
\be \label{phiform}
\Psi_{\fC}(\zz) =\zz^{\frac{\ln(a^{\sigma}) }{\ln(q)}} F_{\fC}(\zz)
\ee
where $F_{\fC}(\zz)$ is holomorphic near $0_{\fC}$ with $F_{\fC}(0_{\fC})=1$ and 
$
a^{\sigma}=M_{\sigma}(0_{\fC}). 
$
\end{Proposition}
 
\begin{proof}
Substituting (\ref{phiform}) to (\ref{qde}) we see that it is consistent  near $0_{\fC}$ and coefficients $c_{\dd}$ in (\ref{fholm}) are fixed uniquely by $c_{0}=1$.
\end{proof}
We call $F_{\fC}(\zz)$ the {\it analytic part} of the solution $\Psi_{\fC}(\zz)$. 

\begin{Definition}
The $q$-periodic function
$$
R_{\fC_1 \leftarrow \fC_2}(\zz)=\Psi_{\fC_1}(\zz)  \Psi^{-1}_{\fC_2}(\zz) 
$$
is called the monodromy of the qde from asymptotic zone $\fC_2$ to zone $\fC_1$.
\end{Definition}

\subsection{}
Now, let $\bK=(\matC^{\times})^{|I|}$ be a K\"ahler torus of $X_{\lambda}$. From (\ref{prform}) we have:
\begin{Proposition}
The vertex function of $X_{\lambda}$  satisfies the $q$-difference equation:
\be \label{verqde}
\textbf{V}_{\lambda}(z_{\square} q) =  \dfrac{1-z_{\square}}{1-z_{\square} \hbar} \textbf{V}_{\lambda}(z_{\square})
\ee
\end{Proposition}
This is a special case of the qde's associated with quiver varieties discussed in \cite{OS}.

Each $z_{\square}$ corresponds to a $\bK$-character $\sigma_{\square}$. The complement of the hyperplanes 
$\sigma_{\square}^{\perp}$
\be \label{chams}
\Lie_{\matR}(\bK) \setminus \{\sigma_{\square}^{\perp}:\, \square\in \lambda\} = \coprod \fC
\ee
is the union of chambers defining asymptotic zones of (\ref{verqde}). The corresponding compactification
$\overline{\bK}$ is sometimes called {\it K\"ahler moduli space} of the quiver variety $X_{\lambda}$.
The vertex function (\ref{prform}) is the unique solution of the qde (\ref{verqde}) for the chamber defining the positive stability condition:
$$
\fC_{+} = \{\theta \in \Lie_{\matR}(\bK): \langle \theta, \sigma_{\square} \rangle >0, \square\in \lambda \}. 
$$
The solutions for other chambers are easy to describe.
\begin{Proposition} \label{propv}
The solution of (\ref{verqde}) corresponding to a chamber $\fC$ equals:
\be \label{vercham}
{\Psi}_{\lambda,\fC}(\zz)=\zz^{\frac{\ln( \hbar_{\fC})}{\ln(q)}} \prod_{\substack{{\square\in \lambda,} \\ {\langle \sigma_\square,\fC \rangle >0}}}  \prod\limits_{i=0}^{\infty} \dfrac{1- \hbar z_{\square} q^{i}}{1- z_{\square} q^{i}} \prod_{\substack{{\square\in \lambda,}\\ {\langle \sigma_\square ,\fC \rangle <0}}}  \prod\limits_{i=1}^{\infty} \dfrac{1-  z_{\square}^{-1} q^{i}}{1- z_{\square}^{-1} \hbar^{-1} q^{i}}  
\ee
where $\zz^{\frac{\ln( \hbar_{\fC})}{\ln(q)}}$ denotes the weight determined by the transformations
$$
z_{\square} \to z_{\square} q  \ \ \implies \ \  \zz^{\frac{\ln( \hbar_{\fC})}{\ln(q)}} \to \zz^{\frac{\ln( \hbar_{\fC})}{\ln(q)}} \hbar^{p_{\fC}}     
$$
where $p_{\fC} = |\{ \square\in \lambda: \langle \sigma_{\square},\fC \rangle <0 \}|  \in \matZ$. 
\end{Proposition}

\begin{proof}
It is elementary to check that this function solves 
(\ref{verqde}). It is also clear that the analytic part of $\Psi_{\lambda,\fC}(\zz)$ is holomorphic near $0_{\fC} \in \overline{\bK}$ 
\end{proof}

\subsection{} 
Let us discuss the geometric meaning of the solutions described in Proposition~\ref{propv}. Recall that to define a Nakajima quiver variety one needs to specify a {\it stability} condition for the geometric invariant theory quotient. The stability condition is specified by a choice of $\theta \in \Lie_{\matR}(\bK)$. The corresponding quiver variety changes (by a symplectic flop) when $\theta$ crosses certain hyperplanes in $\Lie_{\matR}(\bK)$. The complement of these hyperplanes divides the space into a set of chambers $\fC$. The quiver variety obtained from a choice of cocharacter depends only on the chamber that contains it.

The vertex function for a quiver variety $X_{\fC}$ formed by a stability condition from a chamber $\fC$ is given by a power series over the degrees of effective curves, which are given by $\fC$:
$$
\textbf{V}_{X_\fC}(\zz)=1+\sum\limits_{\langle \dd,\fC \rangle >0 } c_{\dd} \zz^{\dd} 
$$
The relation between the vertex functions for different stability conditions is described by the following idea:
\begin{Conjecture} \label{stabcong}
The quantum difference equations for quiver varieties are invariant under a change of the stability condition.
\end{Conjecture}

In the case of equivariant cohomology, this conjecture was proven in \cite{MO}. 

This conjecture implies that the vertex functions for  $X_{\fC}$  are solutions of {\it the same qde (independent of $\fC$), analytic near different points of the K\"ahler moduli space}.

Let $X_{\lambda,\fC}$ be the quiver variety obtained from the dimension data given by $\lambda$ with the stability condition $\fC$. 

\begin{Corollary} 
If Conjecture \ref{stabcong} holds then the analytic part of (\ref{vercham}) is the vertex function of the quiver variety $X_{\lambda,\fC}$. 
\end{Corollary} 

In the rest of this paper we assume that Conjecture \ref{stabcong} holds.

\subsection{} 

\begin{Proposition}
The monodromy of (\ref{verqde}) from $\fC_2$ to $\fC_1$ equals:
$$
R_{\fC_1 \leftarrow \fC_2}(\zz)= (-\hbar^{1/2})^{p_{\fC_1}-p_{\fC_{2}}} \zz^{\frac{\ln(\hbar_{\fC_1})}{\ln(q)}-\frac{\ln(\hbar_{\fC_2})}{\ln(q)}} \dfrac{\prod\limits_{\substack{{\square \in \lambda}\\ {\langle \sigma_\square, \fC_1\rangle <0}}} \vartheta(z_{\square})\prod\limits_{\substack{{\square \in \lambda}\\ {\langle \sigma_\square, \fC_1\rangle >0}}}  \vartheta(z_{\square} \hbar)}{\prod\limits_{\substack{{\square \in \lambda}\\ {\langle \sigma_\square, \fC_2\rangle <0}}} \vartheta(z_{\square})\prod\limits_{\substack{{\square \in \lambda}\\ {\langle \sigma_\square, \fC_2\rangle >0}}}  \vartheta(z_{\square} \hbar)}
$$
where 
$$
\vartheta(z)=(z^{1/2}-z^{-1/2}) \prod\limits_{i=1}^{\infty} (1-z q^{i}) (1-z^{-1} q^{i})
$$
denotes the odd Jacobi theta function. 
\end{Proposition}
\begin{proof}
By definition 
$$
R_{\fC_1 \leftarrow \fC_2}(\zz)=\dfrac{\Psi_{\lambda,\fC_1}(\zz)}{\Psi_{\lambda,\fC_2}(\zz)}
$$
and the result follows immediately from (\ref{vercham}).
\end{proof}

By Theorem \ref{tRS}, the descendent vertex function differs from vertex function with trivial descendent insertion by a rational function. Taking the analytic parts of (\ref{vercham}) we obtain:
\begin{Proposition} \label{charpro}
$$
\textbf{V}^{(\tau)}_{\lambda,\fC_1}(\zz)=\mathcal{R}_{\fC_1\leftarrow \fC_2}(\zz) \textbf{V}^{(\tau)}_{\lambda,\fC_2}(\zz)
$$
where 
\be \label{monR}
\mathcal{R}_{\fC_1, \fC_2}(\zz)=(-\hbar^{1/2})^{p_{\fC_1}-p_{\fC_{2}}}  \dfrac{\prod\limits_{\substack{{\square \in \lambda}\\ {\langle \sigma_\square, \fC_1\rangle <0}}} \vartheta(z_{\square})\prod\limits_{\substack{{\square \in \lambda}\\ {\langle \sigma_\square, \fC_1\rangle >0}}}  \vartheta(z_{\square} \hbar)}{\prod\limits_{\substack{{\square \in \lambda}\\ {\langle \sigma_\square, \fC_2\rangle <0}}} \vartheta(z_{\square})\prod\limits_{\substack{{\square \in \lambda}\\ {\langle \sigma_\square, \fC_2\rangle >0}}}  \vartheta(z_{\square} \hbar)}.
\ee

\end{Proposition}
\subsection{}
Let $X^{!}_{\lambda}$ denote the symplectic dual variety of $X_{\lambda}$. As explained in Section 4.7 of \cite{dinksmir2},  $X^{!}_{\lambda}\cong T^{*} \matC^{|\lambda|}$ equipped with an action of the torus $\bK\times \matC^{\times}_{\hbar}$, where the second factor acts by scaling the symplectic form with weight $\hbar$. The $\bK$-fixed set of $X^{!}_{\lambda}$ consists of a single point $p$ (the origin of  $\matC^{|\lambda|}$). The character of the tangent space was computed in Proposition 4.9.1 of \cite{dinksmir2}:
\be \label{tsp}
\textrm{char}(T_p X^{!}_{\lambda})=\sum\limits_{\square\in \lambda} \, z_{\square} + z^{-1}_{\square} \hbar^{-1}.
\ee
Thus the chambers (\ref{chams}) are the equivariant chambers of the symplectic dual variety\footnote{By definition, these chambers are connected components of the complement of hyperplanes $w^{\perp} \subset \Lie_{\matR}(\bK)$ where $w$ runs over the weights appearing in normal bundles to fixed components, see Section 9.1.2 in \cite{pcmilect}}.  In view of this observation the previous proposition can be reformulated in the language of the elliptic stable envelopes 
\cite{AOElliptic}.
\begin{Proposition} \label{Rmatprop}
\be \label{mond}
\mathcal{R}_{\fC_1\leftarrow \fC_2}(\zz)=\hbar^{(p_{\fC_1}-p_{\fC_2})/2}\, \mathcal{R}^{X^{!}_{\lambda}}_{\fC_1,\fC_2}(\zz) 
\ee
where $\mathcal{R}^{X^{!}_{\lambda}}_{\fC_1,\fC_2}(\zz)$ is the elliptic R-matrix of the symplectic dual variety $X^{!}_{\lambda}$:
$$
\mathcal{R}^{X^{!}_{\lambda}}_{\fC_1,\fC_2}(\zz):=\textrm{Stab}^{-1}_{\fC_2} \circ  \textrm{Stab}_{\fC_1},
$$
and $\textrm{Stab}_{\fC}$ denotes the elliptic stable envelope of $X^{!}_{\lambda}$ for an equivariant chamber~$\fC$.
\end{Proposition}
\begin{proof}
The elliptic stable envelope of $X^{!}_{\lambda}$ is defined by a set of axioms, see Section 3 in \cite{AOElliptic}. In the case of a finite fixed point set these conditions were explained in Section 2.13 of \cite{SmirnovElliptic}. The diagonal restrictions of the elliptic stable envelope are described there by formula (21). For $X^{!}_{\lambda}$ there is only one fixed point and the character of the tangent space is given by (\ref{tsp}). Thus,
$$
\textrm{Stab}_{\fC}=\prod\limits_{\substack{{\square \in \lambda}\\ {\langle \sigma_\square, \fC\rangle <0}}} \vartheta(z_{\square})\prod\limits_{\substack{{\square \in \lambda}\\ {\langle \sigma_\square, \fC\rangle >0}}}  \vartheta\left(\dfrac{1}{z_{\square} \hbar}\right)=(-1)^{p_{\fC}} \prod\limits_{\substack{{\square \in \lambda}\\ {\langle \sigma_\square, \fC\rangle <0}}} \vartheta(z_{\square})\prod\limits_{\substack{{\square \in \lambda}\\ {\langle \sigma_\square, \fC\rangle >0}}}  \vartheta(z_{\square} \hbar) 
$$
and the proposition follows from (\ref{monR}). 
\end{proof}
This result generalizes the relation between vertex functions of zero-dimensional varieties and characters of tangent spaces of symplectic dual varieties which we discussed in \cite{dinksmir}.

\section{Characters of tautological bundles over $X_{\lambda}$} 
\subsection{} 

Let $\tau \in \textrm{char}(G_{\mathsf{v}})$ which we understand as a symmetric polynomial  $\tau({\bs{x}})$ in the Grothendieck roots ${\bs{x}}=\{x_{\square}\}_{\square\in \lambda}$ of the tautological bundles. 

The associated tautological bundle over the quiver variety $X_{\lambda,\fC}$ defines a Laurent polynomial:
$$
\tau({\bs{x}}_{\fC}) \in K_{\matC^{\times}_{\hbar}}(X_{\lambda,\fC}) = \matQ[\hbar^{\pm 1}] 
$$
where ${\bs{x}}_{\fC} =\{ \hbar^{m_{\fC}(\square)} \}_{\square\in \lambda}$ for some $m_{\fC}(\square) \in \matZ$. 

For the positive stability condition $\theta_{+}$ the integers $m_{\fC}(\square)$ are easy to compute, see Section 2.6 in \cite{dinksmir2}. 
For a general $\fC$, we can analyze the stability conditions as described in Proposition 5.1.5 of \cite{GinzburgLectures}. This is, however, an indirect description, and it is not obvious how to compute the characters  ${\bs{x}}_{\fC}$ in this approach. In this section, we derive an explicit combinatorial formula for ${\bs{x}}_{\fC}$ from  the properties of vertex functions discussed  previously.

\subsection{} 
The computation of ${\bs{x}}_{\fC}$ is based on the following two simple results. 
\begin{Proposition} \label{pro1}
Let $\hat{\textbf{V}}^{(\tau)}_{\lambda,\fC_1}(\zz)$ and $\hat{\textbf{V}}^{(\tau)}_{\lambda,\fC_2}(\zz)$ be capped vertex functions of the quiver varieties $X_{\lambda,\fC_1}$, and $X_{\lambda,\fC_2}$, respectively. Then
$$
\hat{\textbf{V}}^{(\tau)}_{\lambda,\fC_1}(\zz)=\hat{\textbf{V}}^{(\tau)}_{\lambda,\fC_2}(\zz) \in \matQ(\zz,q,\hbar)
 $$
\end{Proposition}
\begin{proof}
From the Proposition \ref{charpro} we see that:
$$
\hat{\textbf{V}}^{(\tau)}_{\lambda,\fC_1}(\zz)=\dfrac{\textbf{V}^{(\tau)}_{\lambda,\fC_1}(\zz)}{\textbf{V}^{(1)}_{\lambda,\fC_1}(\zz)} = \dfrac{\textbf{V}^{(\tau)}_{\lambda,\fC_2}(\zz)}{\textbf{V}^{(1)}_{\lambda,\fC_2}(\zz)}=\hat{\textbf{V}}^{(\tau)}_{\lambda,\fC_2}(\zz)
$$
For the positive stability condition the capped vertex function is described by Theorem \ref{tRS} and is obviously rational.
\end{proof}
\begin{Proposition} \label{pro2}
The capped vertex function $\hat{\textbf{V}}^{(\tau)}_{\lambda,\fC}(\zz)$ has the following expansion:
$$
\hat{\textbf{V}}^{(\tau)}_{\lambda,\fC}(\zz)=\tau(\bs{x}_{\fC}) + \sum\limits_{\langle \dd, \fC \rangle>0} c_{\dd} \zz^{\dd}  
$$
where $c_{\dd}\in\mathbb{Q}(q,\hbar)$.
\end{Proposition}
\begin{proof}
By definition, the capped vertex function is  a power series over degrees of quasimaps in the effective cone determined by $\fC$:
$$
\hat{\textbf{V}}^{(\tau)}_{\lambda,\fC}(\zz)= \sum\limits_{\langle \dd, \fC \rangle\geq 0} c_{\dd} \zz^{\dd}  
$$
The quasimaps of degree zero are trivial which means  $\qm^{0}_{\textrm{nonsing} \, p_2} =X_{\lambda}$, and thus the degree zero coefficient in this expansion is $\tau(\bs{x}_{\fC}) \in K_{\bT}(X_{\lambda})$.
\end{proof}

\begin{Corollary}
The capped vertex function for $X_{\lambda,\fC_+}$ where $\fC_+$ is the chamber corresponding to $\theta_+$ satisfies
$$
\tau(\bs{x}_{\fC})=\hat{\textbf{V}}^{(\tau)}_{\lambda}(0_{\fC}). 
$$
\end{Corollary}
\begin{proof}
This follows immediately from Proposition \ref{pro1} and Proposition \ref{pro2}. 
\end{proof}

\subsection{} 
Let $\sigma_{\square}$ be as before. Let $\sigma_{\square,\square'}$ be the $\bK$-character corresponding to $\frac{\zeta_{\square}}{\zeta_{\square'}}$. We define  $m_{\fC}(\square) \in \matZ$ by
$$
m_{\fC}(\square)=| \{ \square' \in 
C_{\lambda}(\square):\square' \neq \square, \langle \sigma_{\square,\square'},\fC \rangle>0   \}|-|\{\square' \in S_{\lambda}(\square) : \langle \sigma_{\square'},\fC \rangle<0 \}|.
$$
\begin{Theorem}
$$
\tau(\bs{x}_{\fC})=\tau(\{
\hbar^{m_{\fC}(\square)}\} ). 
$$
\end{Theorem}
\begin{proof}
The character $\tau(\bs{x}_{\fC})$ is given by substitution of the Grothendieck roots by some monomials $x_{\square}\to \hbar^{m_{\fC}(\square)}$. As $\tau(\bs{x})$ are symmetric in the Grothendieck roots of the tautological bundles, it is enough to prove the proposition for the $n$-th tautological bundles corresponding to the polynomials:
$$
\tau_{n}(\bs{x})=\sum\limits_{\square\in C_{\lambda}(n)} x_{\square}. 
$$
By definition, for $\theta\in \fC$ 
$$
\hat{\textbf{V}}^{(\tau_{n}(\bs{x}))}_{\lambda}(0_{\fC})=\lim\limits_{w\to 0} \hat{\textbf{V}}^{(\tau_{n}(\bs{x}))}_{\lambda}(\theta(w))
$$
where $w$ denotes coordinate on $\matC^{\times}$ and $\hat{\textbf{V}}^{(\tau_{n}(\bs{x}))}_{\lambda}(\zz)$ denotes the capped vertex function for $X_{\lambda,\fC_+}$.
By Corollary \ref{cor}, we have 
\begin{align*}
  \lim_{w\to 0}  \hat{\textbf{V}}_{\lambda}^{(\tau_n(\bs{x}))}(\theta(w))= \sum_{\substack{\square \in C_{\lambda}(n)}} \prod_{\substack{\square' \in C_{\lambda}(n) \\ \square' \neq \square \\ \frac{\zeta_{\square}}{\zeta_{\square'}}\to 0}} \hbar \prod_{\substack{\square' \in S_{\lambda}(\square) \\ z_{\square'}\to \infty}} \frac{1}{\hbar}
\end{align*}
where $z_{\square'}\to \infty$, $\frac{\zeta_{\square}}{\zeta_{\square'}}\to 0$  describes the behaviour of these weights as $w\to 0$. 
These limits depend only on $\fC$ and these conditions are equivalent to $\langle \sigma_{\square'},\fC \rangle<0$ and $\langle \sigma_{\square,\square'},\fC \rangle>0$ respectively.
\end{proof}
 
\printbibliography

\newpage

\vspace{12 mm}

\noindent
Hunter Dinkins\\
Department of Mathematics,\\
University of North Carolina at Chapel Hill,\\
Chapel Hill, NC 27599-3250, USA\\
hdinkins@live.unc.edu

\vspace{3 mm}

\noindent
Andrey Smirnov\\
Department of Mathematics,\\
University of North Carolina at Chapel Hill,\\
Chapel Hill, NC 27599-3250, USA\\
Steklov Mathematical Institute \\
of Russian Academy of Sciences, \\
Gubkina str. 8, Moscow, 119991, Russia \\
asmirnov@email.unc.edu

\end{document}